%% file: main.tex
\documentclass[english]{amsart}
\usepackage[utf8]{inputenc}

\usepackage{geometry}
\usepackage{babel}
\usepackage{mathtools}
\usepackage{multirow}
\usepackage{amstext}
\usepackage{amsthm}
\usepackage{amssymb}
\usepackage[unicode=true,pdfusetitle,
 bookmarks=true,bookmarksnumbered=false,bookmarksopen=false,
 breaklinks=false,pdfborder={0 0 1},backref=false,colorlinks=false]
 {hyperref}
\makeatletter
\numberwithin{equation}{section}
\numberwithin{figure}{section}
\theoremstyle{plain}
\newtheorem{thm}{\protect\theoremname}[section]
\theoremstyle{plain}
\newtheorem{lem}[thm]{\protect\lemmaname}
\theoremstyle{plain}
\newtheorem{cor}[thm]{\protect\corollaryname}
\theoremstyle{plain}
\newtheorem{prop}[thm]{\protect\propositionname}
\theoremstyle{definition}
\newtheorem{defn}[thm]{\protect\definitionname}
\theoremstyle{remark}
\newtheorem{rem}[thm]{\protect\remarkname}
\theoremstyle{definition}

\usepackage{subcaption}
\usepackage{graphicx}
\usepackage{xcolor}
\usepackage{tikz}
\usepackage{pgfplots}
\usetikzlibrary{positioning}

\global\long\def\R{\mathbb{R}}%
\global\long\def\Q{\mathbb{Q}}%
\global\long\def\Z{\mathbb{Z}}%
\global\long\def\N{\mathbb{N}}%
\global\long\def\M{\mathcal{M}}%
\global\long\def\U{\mathbb{U}}%
\global\long\def\norm#1#2{\left\Vert #1\right\Vert _{#2}}%
\global\long\def\diff{\,d}%
\global\long\def\tendsto#1#2{\underset{#1\to#2}{\longrightarrow}}%
\global\long\def\weakto#1#2{\underset{#1\to#2}{\rightharpoonup}}%
\global\long\def\A{\mathcal{A}}%
\global\long\def\B{\mathcal{B}}%
\global\long\def\Acc#1#2{\underset{#1}{\textrm{Acc}}( #2 )}
\global\long\def\define#1{\textit{#1}}
\global\long\def\Ml{\M_0^l}%
\global\long\def\Ms{\M_0^s}%
\global\long\def\MF{\M_0}%
\global\long\def\Bo{\mathfrak{B}}%
\global\long\def\boldepsilon{\boldsymbol{\varepsilon}}%
\makeatother

\providecommand{\corollaryname}{Corollary}
\providecommand{\definitionname}{Definition}
\providecommand{\examplename}{Example}
\providecommand{\lemmaname}{Lemma}
\providecommand{\propositionname}{Proposition}
\providecommand{\remarkname}{Remark}
\providecommand{\theoremname}{Theorem}

\usepackage{subcaption}
\usepackage{graphicx}
\usepackage{xcolor}
\usepackage{tikz}
\usepackage{pgfplots}
\usepackage{stmaryrd}
\usetikzlibrary{positioning}

\makeatother

\providecommand{\corollaryname}{Corollary}
\providecommand{\definitionname}{Definition}
\providecommand{\examplename}{Example}
\providecommand{\lemmaname}{Lemma}
\providecommand{\propositionname}{Proposition}
\providecommand{\remarkname}{Remark}
\providecommand{\theoremname}{Theorem}

\usepackage{marvosym,fancybox}

\begin{document}

\title[Zero-noise limits measures of perturbed cellular automata]{Characterization of the set of zero-noise limits measures of perturbed cellular automata}

\author{Marsan~Hugo}
\address{IMT, Université Toulouse III - Paul Sabatier, Toulouse, France}
\email{hugo.marsan@math.univ-toulouse.fr}

\author{Sablik~Mathieu}
\address{IMT, Université Toulouse III - Paul Sabatier, Toulouse, France}
\email{msablik@math.univ-toulouse.fr}
\urladdr{https://www.math.univ-toulouse.fr/\~{}msablik/index.html}

\begin{abstract}
We add small random perturbations to a cellular automaton and consider the one-parameter family $(F_\epsilon)_{\epsilon>0}$  parameterized by $\epsilon$ where $\epsilon>0$ is the level of noise. The objective of the article is to study the set of limiting invariant distributions as $\epsilon$ tends to zero denoted $\Ml$. Some topological obstructions appear, $\Ml$ is compact and connected, as well as combinatorial obstructions as the set of cellular automata is countable: $\Ml$ is $\Pi_3$-computable in general and $\Pi_2$-computable if it is uniformly approached. Reciprocally, for any set of probability measures $\mathcal{K}$ which is compact, connected and  $\Pi_2$-computable, we construct a cellular automaton whose perturbations by an uniform noise admit $\mathcal{K}$ as the zero-noise limits measure and this set is uniformly approached. To finish, we study how the set of limiting invariant measures can depend on a bias in the noise. We  construct a cellular automaton which realizes any connected compact set (without computable constraints) if the bias is changed for an arbitrary small value. In some sense this cellular automaton is very unstable with respect to the noise.

\end{abstract}

\maketitle
\tableofcontents{}

\section{Introduction}

A cellular automaton is a complex system defined by a local rule which acts synchronously and uniformly on the configuration space $\A^\Z$, where $\A$ is a finite alphabet. It can be also defined as a continuous function on $\A^\Z$ which commutes with the shift. These simple models have a wide variety of different dynamical behaviours. Cellular automata are used to model phenomena defined by local rules or to implement massively parallel computations. The set of invariant probability measures of the system, denoted by $\M_0$, gives information on the statistical properties of the system. 

Since physical models have some uncertainty, it is natural to perturb a dynamical system at each iteration. From the computer science point of view, these perturbations allow to explore how the model is robust at some noise. A natural way to define a perturbation is to consider that after each iteration of a given cellular automaton, we randomly modify each cell independently with a probability $\epsilon$ where the new letter is chosen uniformly: we call this the uniform perturbation. The set of stationary probability measures of this randomly perturbed dynamics is denoted by $\M_\epsilon$. For many classes of cellular automata, $\M_\epsilon$ is a singleton~\cite{MST19}. It was historically more difficult to find cellular automata such that $\M_\epsilon$ is not a singleton~\cite{Toom80} and the question was very challenging for one dimensional cellular automata~\cite{Gacs01}.

In this article, we consider the set of zero-noise limits of invariant measures of the perturbed dynamical systems as $\epsilon$ tends to zero, denoted by $\Ml$. Its corresponds to the set of accumulation points of the sets $\M_\epsilon$ when $\epsilon$ goes to $0$. $\Ml$ is actually a subset of $\M_0$ and can be seen as a way to select physically relevant invariant measures of the system. Informally it is an invariant measure which is supposed to be observable even if there is some noise. 

Usually, one can look for conditions on the dynamical system to force $\Ml$ to be a singleton: the dynamical system is then said to be strongly stochastically stable. In the context of smooth dynamical systems. L.S. Young proves the stochastic stability of $\mathcal{C}^2$ hyperbolic diffeomorphism which comes from the persistence of hyperbolic attractors ~\cite{Young-86}. This problematic was explored for different class of dynamical systems~\cite{Alves-Arau-Vasquez-2007,AlvesVilarinho-2013} sometimes using specific methods as spectral properties of transfer operators~\cite{Viana_97}. See~\cite{Young02} for an overview. In the context of cellular automata the authors have proven the strong stochastic stability for some classes of cellular automata and the undecidability of this property~\cite{Marsan-Sablik-2023}.

The objective of this article is to characterize which set of invariant measures can be obtained as $\Ml$ for a cellular automaton perturbed by a uniform noise. In particular we exhibit a wide variety of sets which can be reached as $\Ml$. The same question is also relevant for other classes of dynamical systems. 

Note that for large classes of cellular automata $\M_\epsilon$ is a singleton~\cite{MST19}: in this case if $\Ml$ is not a singleton, the stationary probability measure of the randomly perturbed dynamics associated can oscillate between different measures when $\epsilon$ goes to $0$. We can generalize this phenomenon: $\Ml$ is said to be uniformly approached if for any family $(\pi_\epsilon)_{\epsilon>0}$ of measures such that $\pi_\epsilon\in\M_\epsilon$ the adherence values of the family is the whole $\Ml$. In other words, the description of $\Ml$ does not depend of the family of invariant measures chosen. If $\Ml$ contains at least two elements and is uniformly approached, the system is chaotic in the sense that there is no stable measure. A stable measure $\mu$ is an invariant measure such that there exists a family of measures $(\pi_\epsilon)_{\epsilon>0}$ such that $\pi_\epsilon\in\M_\epsilon$ and $(\pi_\epsilon)_{\epsilon>0}$ converges toward $\mu$ when $\epsilon$ goes to $0$. In this article we construct sets of zero-noise limits  measures which are uniformly approached and so the cellular automata constructed are chaotic when $\Ml$ is not a singleton. This notion of chaocity can be compared to the chaocity studied in~\cite{ChaHo10} in the context of Gibbs measures defined by local range: when the temperature goes to $0$ the Gibbs measures oscillate between two states of the ground states (measures obtained as adherence values of Gibbs measures when the temperature goes to $0$). In~\cite{GST23} the notion of ground states uniformly approached is introduced in view to characterize them geometrically. 

To characterize sets realized as some $\Ml$, natural topological obstructions appear and are exhibited in Section~\ref{section.TopoContraint}. The set $\Ml$ is closed (so compact) as adherence values in the set of probability measures and is also connected. As the class of cellular automata is countable, there is also combinatorial obstructions which must appear to describe the possible zero-noise limits of invariant measures sets. Recently, a lot of dynamical invariants of cellular automata or subshift of finite type were characterized using computability tools (possible entropies~\cite{Hochman-Meyerovitch-2010}, possible growth-type invariants~\cite{Meyerovitch-2011}, possible sub-actions~\cite{Hochman-2009,Aubrun-Sablik-2010}, possible limit measures~\cite{HS18}...). Thus it is natural to search such obstructions in this context. In Section~\ref{section.ComputContraint} we introduce some notion of computability. A way to describe the algorithmic complexity of a compact set is to consider the next decision problem: given a ball in input, we ask if it intersects the compact set. As we consider asymptotic behavior, this decision problem is undecidable so we need to use the arithmetical hierarchy~\cite{Rogers-1987} to precise the level of uncomputability of this decision problem. In this setting, $\Ml$ is $\Pi_3$-computable and if $\Ml$ is uniformly approached it is $\Pi_2$-computable. 

A difficult question is to determine if these obstructions are the only ones. In Section~\ref{Section.TheoremeRealisation} we prove the main result of this article (Theorem~\ref{thm:realization}) which says that given $\mathcal{K}$ a $\Pi_2$-computable connected and compact subset of shift-invariant probability measures on $\A^\Z$, there exists a cellular automaton on $\B^\Z$ where $\A\subset\B$ such that the associated set $\Ml$ is $\mathcal{K}$ and $\Ml$ is uniformly approached. Note that the cellular automata cannot be surjective since in~\cite{MST19} it is proven that in this case $\M_\epsilon$ is reduced to the uniform Bernoulli measure and therefore so is $\Ml$. Another remark is that $\Ml$ is uniformly approached, which means that the stationary probability measures of the randomly perturbed dynamics associated $\M_{\epsilon}$ oscillate when $\epsilon$ goes to $0$.

The proof of Theorem~\ref{thm:realization} adapts the construction proposed in~\cite{HS18} which realizes the set of measures reached asymptotically when an initial measure is  iterated by a deterministic cellular automaton. The main idea is that there is a special symbol which appears only in the initial configuration or after a perturbation by the noise which initializes a computation realizing approximations of the target set. If two distinct zones of the approximation are in competition, only the youngest remains. The difficulty in our context is to understand how the computation implemented in a such construction can resist at the noise or at least why the noise does not affect so much the computation and the approximation of the target set. As the special symbol only appear with the noise, a $\epsilon$ closer to $0$ means that a zone of approximation has more time to compute the target set before being deleted by a younger computation zone, and thus can give a better approximation. 

In the obstructions, the $\Pi_2$-computability comes from the computability of both the dynamical system (the cellular automaton) and the perturbation itself (the uniform noise). It is possible to put some information in the noise introducing some bias which can be uncomputable. In Section~\ref{section.DependenceNoise}, we slightly modify the construction of Theorem~\ref{thm:realization} in order to take as an oracle input the bias of the noise for the computations. This allows to remove the $\Pi_2$-computable assumption in Theorem~\ref{thm:realization}. In Corollary~\ref{Cor.Strongly Unstable} we exhibit a strongly unstable cellular automata in the following sense: for any $\mathcal{K}$ a connected compact set of shift invariant measures (not necessary $\Pi_2$-computable) and any bias, there exists a bias arbitrary near of the first one such that the set $\Ml$ associated is $\mathcal{K}$ and $\Ml$ is uniformly approached.  

We finish this introduction with an open question. In all our constructions the set $\Ml$ is uniformly approached. When $\Ml$ contains a stable measure and it is not reduced to a singleton (in this case $\Ml$ is not uniformly approached), we don't know which set of measures can be reached : this situation is actually more difficult to produce since this implies that $\M_{\epsilon}$ contains at least two probability measures for $\epsilon$ small enough, and the only known example of one-dimensional cellular automaton with this behavior is quite complex~\cite{Gacs01}.

\section{Notion of stability for cellular automata}

\subsection{Deterministic and probabilistic cellular automata}

Let $\A$ be a finite alphabet of symbols, and consider $\A^{\Z}$ the space of configurations of $\Z$ endowed with the product topology. An application $F:\A^\Z\to\A^\Z$ is a \define{cellular automaton} (CA) if there is a finite neighborhood $\mathcal{N}=\left\{ i_{1},...,i_{r}\right\} \subset\Z$ and a local rule $f:\A^{\mathcal{N}}\to\A$ such that for all $i\in\Z$, $\left(Fx\right)_{i}=f(x_{i+\mathcal{N}})$ where $x_{i+\mathcal{N}}=(x_{i+i_{1}},x_{i+i_{2}},...,x_{i+i_{r}})$. Equivalently, a cellular automaton can be defined as a continuous map on $\A^\Z$ which commutes with the shift map $\sigma:\A^\Z\to\A^\Z$ defined by $\sigma(x)_i=x_{i+1}$ for $x\in\A^\Z$ and $i\in\Z$. 

For a fixed finite subset $\U\subset\Z$, a \define{pattern} $u$ is an element of $\A^\U$. Define the cylinder centered on $u$ by $[u]_\U=\{x\in\A^\U:x_\U=u\}$ (just $[u]$ if there is no ambiguity). The set of cylinders is a basis for the product topology and the $\sigma$-algebra engendered by them is the set of Borelian denoted by $\Bo$.

For a \define{probabilistic cellular automaton} (PCA) the local rule is  randomized and independently applied at every site. More specifically, the local rule of a PCA is a stochastic matrix  $\varphi:\A^{\mathcal{N}}\times\A\to[0,1]$ such that $\sum_{b\in\A}\varphi(u,b)=1$ for each $u\in\A^\mathcal{N}$.  Starting from a configuration $x$, the symbol at each site $i$ is updated at random according to the distribution $\varphi(x_{i+\mathcal{N}},\cdot)$ independently of the other sites. This is described by a transition kernel $\Phi$ where
\[
\forall x\in\A^\Z,\forall\U\subset\Z,\forall w\in\A^\U,\quad \Phi(x,[w]_{\U})=\prod_{i\in\U}\varphi(x_{i+\mathcal{N}},w_{i}).
\]
Moreover, $\Phi$ is an \define{$\epsilon$-perturbation} of a CA $F$ if they are defined on the same alphabet, have the same neighborhood, and if their local rules $\varphi$ and $f$ satisfy 
\[
\forall u\in\A^\mathcal{N},\quad\varphi\left(u,f(u)\right)\geq1-\epsilon.
\]

A family of transition kernel $(F_\epsilon)_{\epsilon>0}$ is called a \define{perturbation of a cellular automaton $F$}  if for all $\epsilon>0$, $F_\epsilon$ is  an $\epsilon$-perturbation of $F$. The stochastic matrix associated is then denoted by $f_{\epsilon}$. We can also consider $F$ as a PCA with local rule $f_0$, where $f_0$ is defined for $u\in\A^{\mathcal{N}}$ and $a\in\A$ by $f_0(u,a)=1$ if $f(u)=a$ and $f_0(u,a)=0$ otherwise. The definition of perturbation is then equivalent (after renaming) to the convergence of $f_{\epsilon}$ to the matrix $f_0$ when $\epsilon$ goes to $0$.

\begin{defn}
A usual perturbation is the perturbation by a \define{uniform noise}: after applying the cellular automaton, each cell is modified independently with probability $\epsilon$ by a symbol chosen uniformly in $\A$. In other words, for $u\in\A^{\mathcal{N}}$ and $a\in\A$ one has 
\[
f_{\epsilon}(u,a)=
\begin{cases}
 \frac{\epsilon}{|\A|}&\textrm{ if }f(u)\ne a\\
1-\epsilon+\frac{\epsilon}{|\A|}&\textrm{ if }f(u)=a\\
\end{cases}.
\]
\end{defn}

\subsection{Action on the set of shift-invariant probability measures}

In this article we consider Borel probability measures which are invariant by the shift map, that is to say measures $\mu$ which satisfy $\mu(\sigma^{-1}(A))=\mu(A)$ for any observable $A\in\Bo$. Denote by $\M(\A^\Z)$ the \define{set of shift-invariant probability measures}.  This set is compact, convex and metrizable for the weak convergence topology: $$\mu_{n}\weakto n{\infty}\mu \textrm{
if }\mu_{n}([u]_{\U})\tendsto n{\infty}\mu([u]_{\U})\textrm{ for all cylinders }[u]_{\U}.$$

We endow $\M\left(\A^\Z\right)$ with the distance $d_{\M}$ which generates the weak convergence topology: 
\[
d_{\M}\left(\mu,\nu\right)\coloneqq\sum_{n\in\N}\frac{1}{2^{n}}\norm{\mu-\nu}{\left\llbracket -n,n\right\rrbracket} ,\qquad \textrm{ where }\qquad \norm{\mu-\nu}{\U}\coloneqq\frac{1}{2}\sum_{u\in\A^{\U}}\left|\mu([u])-\nu([u])\right|
\]
is the total variation distance on the finite set $\U\subset\Z$.

A PCA  $\Phi$  acts on  the set of Borel probability measures on $\A^\Z$ by 
\[
\Phi\mu(A) = \int\Phi(x,A)\diff\mu(x) \textrm{ for all }A\in\Bo,
\]
which in the case of a deterministic CA becomes $F\mu(A) = \mu\left(F^{-1}(A)\right)$. A  probability measure is \define{$\Phi$-invariant} if $\Phi\mu=\mu$. Traditionally the set of $\Phi$-invariant probability measure is denoted by $\M_{\Phi}(\A^\Z)$ (recall that in this article the measures considered are also shift-invariant). To simplify the notation, for $(F_{\epsilon})_{\epsilon>0}$ a perturbation of a CA $F$, the set $\M_{F_{\epsilon}}(\A^\Z)$ is just denoted by $\M_{\epsilon}$ and $\M_0$ denotes the set of $F$-invariant measures.

\begin{prop}
Let $\Phi$ be a PCA of radius $r$, that is to say $\mathcal{N}\subset\left\llbracket -r,r\right\rrbracket$. The action of $\Phi$ is $2^{r}$-Lipschitz on $\M(\A^\Z)$ for the distance $d_{\M}$.
\end{prop}

\begin{proof}
To simplify notations, let $I_{n}$ denote the set $\left\llbracket -n,n\right\rrbracket$. Let us first show that $\norm{\Phi\mu-\Phi\nu}{I_{n}}\leq\norm{\mu-\nu}{I_{n+r}}$. Since $\mathcal{N}=\left\{ j_{1},...,j_{k}\right\} \subset I_{r}$, for any $u\in\A^{I_{n}}$, $\Phi\left(x,\left[u\right]\right)=\prod_{i\in I_{n}}\phi\left(x_{i+j_{1}},...,x_{i+j_{k}}\right)\left(u_{i}\right)$ has the same value for all $x\in\left[w\right]$, where $w\in\A^{I_{n+r}}$, let us denote by $\Phi\left(\left[w\right],\left[u\right]\right)$ this value. Using the definition, for any $u\in\A^{I_{n}}$, we obtain

\[
\Phi\mu([u])=\int \Phi\left(x,\left[u\right]\right)\diff\mu(x)=\sum_{w\in\A^{I_{n+r}}}\Phi\left(\left[w\right],\left[u\right]\right)\mu\left(\left[w\right]\right).
\]

We can finally compute
\begin{align*}
\norm{\Phi\mu-\Phi\nu}{I_{n}} & =\frac{1}{2}\sum_{u\in\A^{I_{n}}}\left|\Phi\mu([u])-\Phi\nu([u])\right|\\
 & =\frac{1}{2}\sum_{u\in\A^{I_{n}}}\left|\sum_{w\in\A^{I_{n+r}}}\Phi\left(\left[w\right],\left[u\right]\right)\left(\mu\left(\left[w\right]\right)-\nu\left(\left[w\right]\right)\right)\right|\\
 & \leq\frac{1}{2}\sum_{u\in\A^{I_{n}}}\sum_{w\in\A^{I_{n+r}}}\Phi\left(\left[w\right],\left[u\right]\right)\left|\mu\left(\left[w\right]\right)-\nu\left(\left[w\right]\right)\right|\\
 & =\frac{1}{2}\sum_{w\in\A^{I_{n+r}}}\left|\mu\left(\left[w\right]\right)-\nu\left(\left[w\right]\right)\right|\underset{=1\text{ (probability kernel)}}{\underbrace{\sum_{u\in\A^{I_{n}}}\Phi\left(\left[w\right],\left[u\right]\right)}}\\
 & =\norm{\mu-\nu}{I_{n+r}}.
\end{align*}

To conclude, we have
\[
d_{\M}\left(\Phi\mu,\Phi\nu\right)=\sum_{n\in\N}\frac{1}{2^{n}}\norm{\Phi\mu-\Phi\nu}{I_{n}} \leq\sum_{n\in\N}\frac{1}{2^{n}}\norm{\mu-\nu}{I_{n+r}}
  =2^{r}\sum_{n\in\N}\frac{1}{2^{n+r}}\norm{\mu-\nu}{I_{n+r}} \leq2^{r}d_{\M}\left(\mu,\nu\right).
\]

\end{proof}

\subsection{The zero-noise limits invariant measures}

Given the common occurrence of accumulation points in this article, we will use the following notations: for $(M_\epsilon)_{\epsilon\in D}$ a family of closed sets of $\M(\A^\Z)$ indexed by elements of $D\subset\R_+$ admitting $0$ as an accumulation point, define the set of accumulation points by
\[
\Acc{\epsilon\to 0}{(M_\epsilon)_{\epsilon\in D}}\coloneqq\left\{ \mu \mid \exists\delta>0,\ \forall\epsilon_0>0,\ \exists\epsilon\in D\cap]0,\epsilon_0[,\ \exists \nu\in M_\epsilon \textrm{ such that } d_\M(\mu,\nu)<\delta \right\} 
\]

In other word, $\mu\in \Acc{\epsilon\to 0}{(M_\epsilon)_{\epsilon\in D}}$ if there exists a sequence $(\epsilon_{n})_{n\in\N}$ of elements of $D$ which converges to $0$ such that for all $n$ there exists $\mu_{\epsilon_{n}}\in M_{\epsilon_{n}}$ which satisfy $\mu_{\epsilon_{n}}\weakto n{\infty}\mu$. In general $D=\R_+^\ast$ and we just denote this set by $\Acc{\epsilon\to 0}{M_\epsilon}$. If $M_\epsilon$ is a singleton $\{\pi_{\epsilon}\}$ for all $\epsilon>0$, we just denote it by $\Acc{\epsilon\to 0}{\pi_\epsilon}$.

\begin{defn}
Given a perturbation of a cellular automaton $\left(F_{\epsilon}\right)_{\epsilon>0}$ and $\M_{\epsilon}$ the set of $(F_{\epsilon},\sigma)$-invariant measures,  $\Ml$ is the \define{zero-noise limits of invariant measures} (in short the \define{zero-noise limits measures}) defined by $\Ml=\underset{\epsilon\to0}{\textrm{Acc}}(\M_{\epsilon})$.
\end{defn}

\begin{rem}
The set $\Ml$ depends on the cellular automaton considered and on the perturbation: we only precise them in arguments in ambiguous situations.
\end{rem}

\subsection{Stable measures and chaocity}

A zero-noise limit measure $\nu$ is not necessary stable, in the sense that there could be a $\delta > 0$ and a sequence $(\epsilon_n)_{n\in\N}$ such that the sets $\M_{\epsilon_n}$ do not intersect the ball of center $\nu$ and radius $\delta$.

\begin{defn}
A measure $\nu\in \M$ is \define{stable} if for all positive $\epsilon$ there exists $\pi_{\epsilon}\in\M_{\epsilon}$ such that the family $(\pi_\epsilon)_{\epsilon>0}$ satisfies $\pi_{\epsilon}\weakto{\epsilon}{0}\nu$. The set of stable measures is denoted by $\Ms$.

A perturbation of a cellular automaton $\left(F_{\epsilon}\right)_{\epsilon>0}$ is said to be \define{chaotic} if it does not admit any stable measures. 
\end{defn}

\begin{rem}
A stable measure is a zero-noise limit measure, so $\Ms\subset\Ml$.
\end{rem}

A system is then chaotic if there are at least two zero-noise limits measures such that when $\epsilon$ goes to $0$ the invariant measures of the perturbed system oscillate between these two measures. We want a stronger definition in the sense that the convergence towards $\Ml$ does not depend to the choice of the family of invariant measures.

\begin{defn}
Given a perturbation of a cellular automaton $\left(F_{\epsilon}\right)_{\epsilon>0}$, the set of zero-noise limits of invariant measures $\Ml$ is \define{uniformly approached} if $\Ml=\Acc{\epsilon\to 0}{\pi_{\epsilon}}$ for any choice of the family $\left(\pi_{\epsilon}\right)_{\epsilon>0}$ such that $\pi_{\epsilon}\in\M_{\epsilon}$ for all $\epsilon>0$.
\end{defn}

$\Ml$ can be uniformly approached and still contain at least two measures: in this case $\Ms$ is empty and $\left(F_{\epsilon}\right)_{\epsilon>0}$ is said to be \define{uniformly chaotic}.

\section{Topological constraints}\label{section.TopoContraint}

By definition $\Ml$ and $\Ms$ are closed and thus compact since $\M(\A^\Z)$ is compact. To consider some other topological properties of these sets some regularities on the perturbation of a cellular automaton are required. Let us introduce the notion of continuous perturbation.   

\begin{defn}
A perturbation of a cellular automaton $(F_{\epsilon})_{\epsilon >0}$ is \define{continuous} if the function $\epsilon\longmapsto f_\epsilon$ is continuous for $\epsilon>0$ where $f_\epsilon$ is the local rule (stochastic matrix) of $F_\epsilon$.
\end{defn}

\subsection{Continuity Lemma}

\begin{lem}\label{Lemma-ContinuityLemma}
If $\epsilon\longmapsto f_{\epsilon}$ is continuous at $\epsilon_{0}\geq 0$ and $\pi_{\epsilon}\weakto{\epsilon}{\epsilon_{0}}\pi_{\epsilon_{0}}$ where $\pi_{\epsilon}\in\M_{\epsilon}$ for each $\epsilon\geq 0$, then $F_{\epsilon}\pi_{\epsilon}\weakto{\epsilon}{\epsilon_{0}}F_{\epsilon_0}\pi_{\epsilon_{0}}$. In particular $\Acc{\epsilon\to\epsilon_0}{\M_{\epsilon}}\subset\M_{\epsilon_0}$.
\end{lem}

\begin{proof}
Let $\U\subset\Z$ be a finite set. Let us prove that  $\norm{F_{\epsilon}\pi_{\epsilon}-F_{\epsilon_{0}}\pi_{\epsilon_{0}}}{\U}\tendsto{\epsilon}{\epsilon_{0}}0$.

Decompose $\norm{F_{\epsilon}\pi_{\epsilon}-F_{\epsilon_{0}}\pi_{\epsilon_{0}}}{\U}\leq\underset{2}{\underbrace{\norm{F_{\epsilon}\pi_{\epsilon}-F_{\epsilon_{0}}\pi_{\epsilon}}{\U}}}+\underset{1}{\underbrace{\norm{Fi_{\epsilon_{0}}\pi_{\epsilon}-F_{\epsilon_{0}}\pi_{\epsilon_{0}}}{\U}}}$.
\begin{enumerate}
\item Converges to $0$ by continuity of the action of $\Phi_{\epsilon_{0}}$
on $\M(\A^\Z)$.
\item By definition, it suffices to show that for all $u\in\A^{\U}$, $F_{\epsilon}\pi_{\epsilon}([u]_{A})-F_{\epsilon_{0}}\pi_{\epsilon}([u]_{A})\tendsto{\epsilon}{\epsilon_{0}}0$. One has
\[
\left|F_{\epsilon}\pi_{\epsilon}([u]_{\U})-F_{\epsilon_{0}}\pi_{\epsilon}([u]_{\U})\right| \leq\int\left|F_{\epsilon}(x,[u]_{\U})-F_{\epsilon_{0}}(x,[u]_{\U})\right|\diff\pi_{\epsilon}.
\]
Then observe that
\begin{align*}
F_{\epsilon}(x,[u]_{\U}) & \coloneqq\prod_{i\in \U}f_{\epsilon}\left(x_{i+\mathcal{N}},u_{i}\right)\\
	& \leq \prod_{i\in\U}\left(f_{\epsilon_{0}}\left(x_{i+\mathcal{N}},u_{i}\right) + \Delta_{\epsilon,\epsilon_{0}}\right) & \textrm{where }\Delta_{\epsilon,\epsilon_{0}} = \norm{f_{\epsilon}-f_{\epsilon_0}}{\U}\\
	& \leq \left(\prod_{i\in \U}f_{\epsilon_{0}}\left(x_{i+\mathcal{N}},u_{i}\right)\right) + \sum_{k=1}^{\left|A\right|}\binom{\left|A\right|}{k}\Delta_{\epsilon,\epsilon_{0}}^{k} & \text{by expanding and }f_{\epsilon_{0}}\in\left[0,1\right]\\
	& = F_{\epsilon_{0}}(x,[u]_{A})+\left(\Delta_{\epsilon,\epsilon_{0}}+1\right)^{\left|A\right|}-1
\end{align*}
and so by symmetry of the computations, 
\[
\left|F_{\epsilon}(x,[u]_{A})-F_{\epsilon_{0}}(x,[u]_{A})\right|\leq\left(\Delta_{\epsilon,\epsilon_{0}}+1\right)^{\left|A\right|}-1.
\]
Since $\epsilon\longmapsto f_{\epsilon}$ is continuous at $\epsilon_{0}\geq 0$, $\Delta_{\epsilon,\epsilon_{0}}\tendsto{\epsilon}{\epsilon_{0}}0$ and finally, 
\[
\left|F_{\epsilon}\pi_{\epsilon}([u]_{A})-F_{\epsilon_{0}}\pi_{\epsilon}([u]_{A})\right|\leq\left(\Delta_{\epsilon,\epsilon_{0}}+1\right)^{\left|A\right|}-1\tendsto{\epsilon}{\epsilon_{0}}0.
\]
Hence the result.
\end{enumerate}
\end{proof}

As by definition a perturbation is always continuous at $\epsilon_0=0$, we obtain the following corollary.
\begin{cor}
 Given a perturbation of a cellular automaton $\left(F_{\epsilon}\right)_{\epsilon>0}$, we have $\Ms\subset\Ml\subset\MF$.
\end{cor}

\subsection{Connectedness of $\Ml$}

\begin{prop}\label{prop.MlConnected}
Let $(F_{\epsilon})_{\epsilon >0}$ be a continuous perturbation of a cellular automaton. The set $\Ml$ is connected.
\end{prop}

\begin{proof}
By contradiction, let us suppose that $\Ml=A\bigsqcup B$ with $A,B$ non-empty closed subsets of $\Ml$. $\Ml$ is a compact of $\M(\A^\Z)$ so $A$ and $B$ must be too. We can then define $\alpha=d_\M(A,B)>0$. Then define 
\begin{itemize}
\item $A^{\prime}=\left\{ \nu\in\M\mid d_\M(A,\nu)<\frac{\alpha}{3}\right\} $
which is open.
\item $B^{\prime}=\left\{ \nu\in\M\mid d_\M(B,\nu)<\frac{\alpha}{3}\right\} $
which is open.
\item $K=\left(A^{\prime}\sqcup B^{\prime}\right)^{c}$ which is closed and thus compact.
\end{itemize}

\begin{figure}[h!]
\begin{centering}
\def\svgwidth{\columnwidth*3/4}
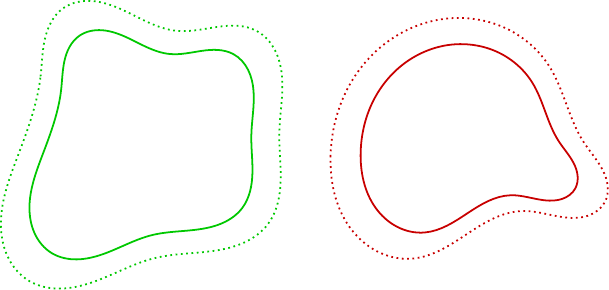
\par\end{centering}
\centering{}\caption{Proof's scheme of Proposition~\ref{prop.MlConnected}.}
\end{figure}

Let $\mu\in A$ and $\nu\in B$. There exist two sequences $(\pi_{\epsilon_n})_{n\in\N}$ and  $(\pi'_{\epsilon'_n})_{n\in\N}$ such that $\pi_{\epsilon_n}\in\M_{\epsilon_n}$, $\pi'_{\epsilon_n}\in\M_{\epsilon'_n}$ for all $n\in\N$ and such that $\pi_{\epsilon_{n}}\weakto n{\infty}\mu$, $\pi'_{\epsilon'_{n}}\weakto n{\infty}\nu$. It is possible to assume that $\epsilon_{n}>\epsilon'_{n}$ and those sequences are strictly decreasing. Define:
\begin{itemize}
\item $V_{A}=\left\{ n\geq 0 \mid d_\M(\mu_{\epsilon_{n}},\mu)\leq\frac{\alpha}{6}\right\} .$
\item $V_{B}=\left\{ n\geq 0 \mid d_\M(\nu_{\epsilon_{n}^{\prime}},\nu)\leq\frac{\alpha}{6}\right\} .$
\end{itemize}
Let $n_{A}\in V_{A}$. As $V_{B}$ is infinite, we can also define $n_{B}\geq n_{A}$ in $V_{B}$. In particular, $\M_{\epsilon_{n_{A}}}\cap A^{\prime}\neq\emptyset$ and $\M_{\epsilon'_{n_B}}\cap B'\neq\emptyset$. We want to exhibit $\epsilon\in\left[\epsilon'_{n_B},\epsilon_{n_A}\right]$ such that $\M_{\epsilon}\cap K\neq\emptyset$. 

By contradiction, assume that $\forall\epsilon\in\left[\epsilon_{n_{B}}^{\prime},\epsilon_{n_{A}}\right]$ we have $\M_{\epsilon}\subset A^{\prime}\sqcup B^{\prime}$. By convexity of $\M_{\epsilon}$, we then have $\M_{\epsilon}\subset A^{\prime}$ or $\M_{\epsilon}\subset B^{\prime}$. Define $E\coloneqq\left\{ \epsilon\in\left[\epsilon_{n_{B}}^{\prime},\epsilon_{n_{A}}\right]\mid\M_{\epsilon}\subset B^{\prime}\right\} $ and $\epsilon_{c}\coloneqq\sup E$ (non-empty set as $\M_{\epsilon_{n_{B}}^{\prime}}\subset B$, and bounded from above by $\epsilon_{n_{A}}$). We can then define the following sequences:
\begin{enumerate}
\item $(e_{n})_{n\in\N}$ such that $e_{n}\to\epsilon_{c}$ and $e_{n}\in E$ for all $n\in\N$, this sequence exists as $\epsilon_{c}$ is a $\sup$. Then take a sequence $\pi_{e_{n}}\in\M_{e_{n}}$ and using sub-sequences of $\left(e_{n}\right)$, we can suppose that $\left(\pi_{e_{n}}\right)$ converges to a $\pi\in\M$. By Lemma~\ref{Lemma-ContinuityLemma}, we have  $\pi\in\M_{\epsilon_{c}}$. But by continuity of the distance, $\pi_{e_{n}}\in B^{\prime}$ for all $n$ implies that $d_\M\left(\pi,B\right)\leq\frac{\alpha}{3}$, and thus $\M_{\epsilon_{c}}\subset B^{\prime}$.
\item $(f_{n})_{n\in\N}$ such that $f_{n}\to\epsilon_{c}$ and $f_{n}\notin E$ for all $n\in\N$, this sequence exists as $\epsilon_{c}$ is a $\sup$. By an analogous reasoning, one obtains $\M_{\epsilon_{c}}\subset A^{\prime}$, and thus a contradiction.
\end{enumerate}

Since $V_A$ is infinite, it is possible to construct an increasing sequence $(\epsilon''_n)_{n\in\N}$ and a sequence $(\pi''_{\epsilon''_n})_{n\in\N}$ such that $\pi''_{\epsilon''_n}\in K\cap\M_{\epsilon''_n}$. By compactness, there exists $\pi\in K\cap\Ml$. This is a contradiction.
\end{proof}

\subsection{Convexity of $\Ms$}
\begin{prop}
The set $\Ms$ of stable measures of a perturbation of a cellular automaton is convex.
\end{prop}
\begin{proof}
 Let $\mu,\nu\in\Ms$. There exists $(\pi_{\epsilon})_{\epsilon>0}$ and $(\pi'_{\epsilon})_{\epsilon>0}$ such that $\pi_{\epsilon}\weakto{\epsilon}{0}\mu$, $\pi'_{\epsilon}\weakto{\epsilon}{0}\nu$ and $\pi_{\epsilon},\pi'_{\epsilon}\in\M_{\epsilon}$. Thus for $t\in[0,1]$, one has $t\pi_{\epsilon}+(1-t)\pi'_{\epsilon}\weakto{\epsilon}{0} t\mu+(1-t)\nu$ and $t\pi_{\epsilon}+(1-t)\pi'_{\epsilon}\in\M_{\epsilon}$ since $\M_{\epsilon}$ is convex. We deduce that $\Ms$ is convex.
\end{proof}

\subsection{Some consequences of uniformity}

In order to bound the algorithmic complexity of $\Ml$ when it is uniformly approached, we need some topological lemmas which precise that it is sufficient to look $\Ml$ as the set of accumulation point of $\M_\epsilon$ when the diameter of $\M_\epsilon$ is arbitrary small. We first give some implications between the different possible properties of a perturbation of a CA.

For a compact set of measure $\mathcal{K}$ define its \define{diameter} as $\text{diam}(\mathcal{K})=\max\{d(\mu,\nu):\mu,\nu\in\mathcal{K}\}$. The maximum exists since $\mathcal{K}$ is compact.

\begin{prop}\label{prop.limInf}
Let $\left(F_{\epsilon}\right)_{\epsilon >0}$ be a perturbation of a CA such that $\Ml$ is uniformly approached. Then 
$$\liminf_{\epsilon\to0}\textrm{diam}\left(\M_{\epsilon}\right)=0$$

that is to say $\forall\delta>0, \forall\epsilon_0>0, \exists\epsilon<\epsilon_0, \text{diam}(\M_\epsilon)<\delta$.
\end{prop}

\begin{proof}
Suppose there exist $\delta>0$ and $\epsilon_0>0$ such that for all $\epsilon<\epsilon_0$, one has $\text{diam}\left(\M_{\epsilon}\right)\geq\delta$. For all $\epsilon<\epsilon_0$, choose $\pi_{\epsilon}$ and $\pi'_{\epsilon}$ in $\M_{\epsilon}$ such that $d_\M\left(\pi_{\epsilon},\pi'_{\epsilon}\right)\geq\delta$. Let $\pi$ be a measure in $\Acc{\epsilon\to0}{\pi_{\epsilon}}$. 
Finally define $\mu_{\epsilon}\in\M_{\epsilon}$
as:
\[
\mu_{\epsilon}=\begin{cases}
\pi_{\epsilon} & \text{if }d_\M\left(\pi_{\epsilon},\pi\right)>\frac{\delta}{2}\\
\pi'_{\epsilon} & \text{if }d_\M\left(\pi_{\epsilon},\pi\right)\leq\frac{\delta}{2}
\end{cases}.
\]
Thus for all $\epsilon<\epsilon_0$, by inverse triangular inequality
\[
d_\M\left(\mu_{\epsilon},\pi\right) \geq \left|d_\M\left(\mu_{\epsilon}, \pi_{\epsilon}\right)-d_\M\left(\pi_{\epsilon},\pi\right)\right| \geq \frac{\delta}{2}.
\]
It follows that $\pi\notin\Acc{\epsilon\to 0}{\mu_{\epsilon}}$, so $\Ml$ is not uniformly approached.
\end{proof}

To sum up, the different properties of the perturbation of a CA satisfy the following implications. 
\begin{prop}
Let $\left(F_{\epsilon}\right)_{\epsilon >0}$ be a perturbation of a CA. One has the following implications:
\begin{center}
\begin{tikzpicture}
 \node (A) at (-1.2,1) {$\M_\epsilon$ is a singleton for all $\epsilon > 0$};
 \node (B) at (-2,-1) {$\Ml$ is a singleton};
 \node (C) at (0,0) {$\displaystyle\lim_{\epsilon\to 0} \text{diam}(\M_\epsilon) = 0$};
 \node (D) at (5,0) {$\Ml$ is uniformly approached};
 \node (E) at (10,0) {$\displaystyle\liminf_{\epsilon\to 0} \text{diam}(\M_\epsilon) = 0$};

\draw (-1,0.5) node[rotate=-45] {$\Longrightarrow$};
\draw (-1,-0.5) node[rotate=45] {$\Longrightarrow$};
\draw (2.3,0) node {$\Longrightarrow$};
\draw (7.7,0) node {$\Longrightarrow$};

\end{tikzpicture} 
\end{center}

\end{prop}

\begin{lem}\label{lemma-RestrictionDdelta}
Let $\delta>0$ and $\left(F_{\epsilon}\right)_{\epsilon >0}$ be a perturbation of a CA such that $\Ml$ is uniformly approached. Define $D_{\delta}=\{\epsilon>0:\textrm{diam}(\M_\epsilon)<\delta\}$. Then $D_{\delta}$ is an open set which contains $0$ as accumulation point. Moreover one has

$$\Ml\left(=\Acc{\epsilon\to 0}{(\M_{\epsilon})_{\epsilon>0}}\right)=\Acc{\epsilon\to 0}{(\M_{\epsilon})_{\epsilon\in D_{\delta}}}=\Acc{\epsilon\to 0}{(\M_{\epsilon})_{\epsilon\in D_{\delta}\cap\Q_+}}.$$
\end{lem}

\begin{proof}
The fact that $0$ is an accumulation point of $D_{\delta}$ is a  direct consequence of of Proposition~\ref{prop.limInf}. 

Let us prove that its complement ${D_{\delta}}^c=\left\{ \epsilon>0\mid\text{diam}\left(\M_{\epsilon}\right)\geq\delta\right\} $ is closed. Choose a sequence $\left(\epsilon_{n}\right)_{n\in\N}$ of ${D_{\delta}}^c$ converging to $\epsilon>0$. For all $n\in\N$, there exist $\pi_n,\pi'_n\in\M_{\epsilon_n}$ such that $d_\M(\pi_n,\pi'_n)\geq\delta$. Taking sub-sequences, we can assume that $(\pi_n)_{n\in\N}$ and $(\pi'_n)_{n\in\N}$ converges respectively to $\pi$ and $\pi'$ and by Lemma~\ref{Lemma-ContinuityLemma} one deduces that $\pi$ and $\pi'$ are in $\M_{\epsilon}$. By continuity of the distance $d(\pi,\pi')\geq\delta$ so $\epsilon\in {D_{\delta}}^c$. Thus ${D_{\delta}}^c$ is closed.

Let us prove by contradiction that $\Ml=\Acc{\epsilon\to 0}{(\M_{\epsilon})_{\epsilon\in D_{\delta}}}$. Suppose that there exists a $\pi\in\Ml\smallsetminus\Acc{\epsilon\to 0}{(\M_{\epsilon})_{\epsilon\in D_{\delta}}}$. Consider $\left(\pi_{\epsilon}\right)_{\epsilon>0}$ be a sequence associated to $(\M_{\epsilon})_{\epsilon>0}$:  by uniformity, $\Acc{\epsilon\to 0}{\pi_{\epsilon}}=\Ml$. As $\pi\notin\Acc{\epsilon\to 0}{(\M_{\epsilon})_{\epsilon\in D_{\delta}}}$, then $\pi\notin\Acc{\epsilon\to 0}{(\pi_{\epsilon})_{\epsilon\in D_{\delta}}}$. By definition, there are $\eta>0$, $\epsilon_{0}>0$ such that $d_\M(\pi_{\epsilon},\pi)>\eta$ for all $\epsilon\in D_{\delta}\cap[0,\epsilon_0]$.

Analogously to the proof of Proposition~\ref{prop.limInf}, for all $\epsilon\notin D_{\delta}$ there exists $\mu_{\epsilon}\in\M_{\epsilon}$ such that $d_\M(\pi,\mu_{\epsilon})\geq\frac{\delta}{2}$ since in this case $\textrm{diam}(\M_{\epsilon})\geq\delta$. Define $\nu_{\epsilon}=\begin{cases} 
	\pi_{\epsilon} & \text{if }\epsilon\in D_{\delta}\\
	\mu_{\epsilon} & \text{if }\epsilon\notin D_{\delta}
\end{cases}$. We have constructed a sequence $\left(\nu_{\epsilon}\right)_{\epsilon>0}$ verifying for all $\epsilon<\epsilon_{0}$ the relation $d_\M\left(\nu_{\epsilon},\pi\right) \geq \min\left(\eta,\frac{\delta}{2}\right)$, so $\pi\notin\Acc{\epsilon\to 0}{\nu_{\epsilon}}$. This contradicts the fact that $\Ml$ is uniformly approached.

Let us prove the last point that is to say that it is possible to obtain $\Ml$ taking the accumulation point in $D_{\delta}\cap\Q$. Suppose there is $\pi\in\Ml\smallsetminus\Acc{\epsilon\to 0}{(\M_{\epsilon})_{\epsilon\in D_{\delta}\cap\Q_+}}$. Then there exists $\eta>0$ and $\epsilon_{0}>0$ such that for all $\epsilon\in D_{\delta}\cap\Q\cap[0,\epsilon_0]$ such that $d_\M\left(\pi,\M_{\epsilon}\right)>\eta$. 

Put $\delta^{\prime}=\min\left(\delta,\frac{\eta}{4}\right)$. Since $\Ml=\Acc{\epsilon\to 0}{(\M_{\epsilon})_{\epsilon\in D_{\delta'}}}$, there is a sequence $(\epsilon_{n})_{n\in\N}$ of element of $D_{\delta^{\prime}}\cap[0,\epsilon_{0}]$ and a sequence $(\pi_{\epsilon_{n}})_{n\in\N}$ such that $\pi_{\epsilon_{n}}\in\M_{\epsilon_{n}}$ for all $n\in\N$ and $\pi_{\epsilon_n}\weakto n{\infty}\pi$. In particular, there is a rank $N$ such that $d_\M\left(\pi_{\epsilon_{N}},\pi\right)<\frac{\eta}{2}$. 

One has $\epsilon_{N}\notin\Q$ since $D_{\delta^{\prime}}\subset D_{\delta}$. Denote by $\left(\gamma_{n}\right)$ a sequence of $D_{\delta^{\prime}}\cap[0,\epsilon_0]\cap\Q$ converging to $\epsilon_{N}$, this sequence exists as $\Q$ is dense in the open set $D_{\delta^{\prime}}$. For all $n\in\N$ consider $\nu_{\gamma_{n}}\in\M_{\gamma_{n}}$, by compactness we can extract a sequence converging to a measure $\nu_{\epsilon_{N}}$ and $\nu_{\epsilon_{N}}\in\M_{\epsilon_{N}}$ by Lemma~\ref{Lemma-ContinuityLemma}. Since $\gamma_{n}\in D_{\delta}\cap\Q\cap[0,\epsilon_0]$, one has $d_\M\left(\pi,\nu_{\gamma_n}\right)>\eta$ and by continuity of the distance we have $d_\M\left(\pi,\nu_{\epsilon_{N}}\right)\geq\eta$.
Then by triangular inequality
\[
d_\M\left(\nu_{\epsilon_{N}},\pi_{\epsilon_{N}}\right) \geq \left|d_\M\left(\nu_{\epsilon_{N}},\pi\right)-d_\M\left(\pi_{\epsilon_{N}},\pi\right)\right| \geq \frac{\eta}{2} \geq 2\delta',
\]
which contradicts $\text{diam}\left(\M_{\epsilon_{n}}\right)<\delta^{\prime}$ as $\epsilon_{N}\in D_{\delta'}$.
\end{proof}

\section{Computability constraints}\label{section.ComputContraint}

Since the set of cellular automata is countable, to characterize the set of zero-noise  limits  measures we can search for discrete constraints and more particularly computability constraints. In this section we recall first some notions of computable analysis before using them to show computability constraints on $\Ml$ and $\Ms$.

\subsection{Computable noise}\label{subsec:computableNoise}

We begin by giving the main definitions of computability on metric spaces. We refer to~\cite{Wei00,BraHerWei08}  for a general introduction. The purpose of this section is to define what is a computable noise.

As is standard in computer science, the informal idea of algorithm can be formalized as a \define{Turing machine}~\cite{Turing36} and a function $f:\N\to\N$ is said to be \define{computable} if there exists an algorithm which, upon input $n$, outputs $f(n)$. This notion naturally extends to countable sets such as $\Q$ that can be explicitly encoded as integers. It is possible to extend the notion of computable reals introduced in~\cite{Turing36} for abstract metric spaces~\cite{GaHoRo10}. For that, let us define a \define{computable metric space} as a triplet $(X,\mathfrak{D},d)$ where $(X,d)$ is a metric space that we assume compact (for this article), $\mathfrak{D}=\left(z_n\right)_{n\in\N}$ is a dense set of $X$ and $d$ is uniformly computable, that is to say there exists a computable $f:\N^3\to\Q$ such that for every $i,j,n\in\N$, $\left|d\left(z_i,z_j\right)-f(i,j,n)\right|\leq 2^{-n}$. An element $x\in X$ is said to be \define{computable} if there is a computable map $f:\N\to\N$ such that $d\left(x,z_{f(n)}\right)\leq 2^{-n}$ for every $n$, whereas $x$ is only \define{limit-computable} if $d\left(x,z_{f(n)}\right)\to 0$ but without any control on the convergence speed.

In this article we consider three types of computable metric spaces that clearly satisfy the definition:
\begin{itemize}
	\item $([0,1],\Q\cap[0,1],d)$ where $d(x,y)=|x-y|$. It can easily be generalized to $[0,1]^d$.
	\item $(\M(\A^\Z),\mathfrak{P},d_{\M})$ where $\M(\A^\Z)$ is the set of shift-invariant probability measures, $\mathfrak{P}$ is the set of shift-invariant probability measures supported by a periodic orbit and $d_{\M}$ is the distance introduced previously, compatible with the weak convergence topology. For $w\in\A^\ast$, denote by $\widehat{\delta_w}$ the shift-invariant probability measure supported by the periodic orbit generated by $w$: we can then describe $\mathfrak{P}$ as $\mathfrak{P}=\{\widehat{\delta_w}:w\in\A^\ast\}$. 
	\item $([0,1]\times\M(\A^\Z), (\Q\cap[0,1])\times\mathfrak{P}, D)$ where $D\left(\left(s,\mu\right),\left(t,\nu\right)\right)=\max\left(\left|t-s\right|,d_{\M}\left(\mu,\nu\right)\right)$. It is simply the product of the two previous examples.
\end{itemize}

A function $F:(X,\mathfrak{D},d)\to (X',\mathfrak{D}',d')$ between computable metric spaces is \define{computable} if there exists two computable functions $a:\N\to\N$ and $b:\N^2\to\mathfrak{D}'$ such that
\begin{itemize}
	\item $d(x,y)\leq a(n)\Longrightarrow d'(F(x),F(y))\leq 2^{-n}$ for all $n\in\N$,
	\item $d(F(z_i),b(i,n))\leq 2^{-n}$ for all $i,n\in\N$.
\end{itemize}

With this setting, a perturbation $(F_\epsilon)_{\epsilon>0}$ of a cellular automaton $F$ is \define{computable} if the function $\epsilon\longmapsto f_\epsilon$ is computable where $f_\epsilon\in [0,1]^{\A^\mathcal{N}\times\A}$ is local rule of $F_\epsilon$. For example, the uniform noise is computable.

\begin{lem}\label{lemma.calculNoise}
If the perturbation $(F_\epsilon)_{\epsilon >0}$ is computable, then $\left(\epsilon,\mu\right)\longmapsto F_{\epsilon}\mu$ is computable.
\end{lem}

\begin{proof}
We have to show the effective computability on the rationals and the uniform computability of the continuity modulus.

 For a given $\epsilon\in\left[0,1\right]\cap\Q$ and $\mu\in\mathfrak{P}$, we want to find $\nu\in\mathfrak{P}$ such that $\norm{F_{\epsilon}\mu-\nu}{I_{n}}$ is close to $0$. The function $\epsilon\mapsto f_\epsilon$ is computable, so the values $F_\epsilon([v],[u])=\prod_{i\in I_n}f_{\epsilon}(v_{i+\mathcal{N}})(u_i)$ are computable for $\epsilon\in\Q$ and all finite words $u\in\A^{I_n}, v\in\A^{I_{n+r}}$. As $\mu\in\mathfrak{P}$, the values $\mu([v])$ are also computable. We have that for all $u\in\A^{I_{n}}$, 
\[
F_{\epsilon}\mu([u])=\sum_{v\in\A^{I_{n}+r}}F_{\epsilon}\left(\left[v\right],\left[u\right]\right)\mu\left(\left[v\right]\right)
\]
and thus is also computable. We just have to enumerate $w\in\A^*$ until we find one such that 
\[
\left|F_\epsilon\mu([u]) - \widehat{\delta_w}([u])  \right| \leq \frac{1}{2^{m}}
\]
for all $u\in\A^{\leq n}$ (we will find one eventually, using the density of $\mathfrak{P}$ and computability of $F_\epsilon \mu([u])$), for any $m$ we want. We then have $\norm{F_{\epsilon}\mu-\widehat{\delta_w}}{I_{n}}\leq \frac{|\A|^{|I_n|}}{2^{m+1}}$.

 For the modulus of continuity, observe that for $\epsilon,\epsilon^{\prime}\in\left[0,1\right]$, $\mu,\mu^{\prime}\in\M(\A^\Z),N\geq1$, we have
\begin{align*}
d_{\M}\left(F_{\epsilon}\mu,F_{\epsilon^{\prime}}\mu^{\prime}\right) & \leq d_{\M}\left(F_{\epsilon}\mu,F_{\epsilon^{\prime}}\mu\right)+d_{\M}\left(F_{\epsilon^{\prime}}\mu,F_{\epsilon^{\prime}}\mu^{\prime}\right)\\
 & \leq\frac{1}{2^{N}}+\sum_{n\leq N-1}\frac{1}{2^{n}}\norm{F_{\epsilon}\mu-F_{\epsilon^{\prime}}\mu}{I_{n}}+2^{r}d_{\M}\left(\mu,\mu^{\prime}\right)
\end{align*}
with $\norm{F_{\epsilon}\mu-F_{\epsilon^{\prime}}\mu}{I_{n}}\leq\frac{1}{2}\sum_{u\in\A^{I_{n}}}\sum_{v\in\A^{I_{n+r}}}\left|\prod_{i\in I_{n}}f_{\epsilon}\left(v_{i+\mathcal{N}},u_{i}\right)-\prod_{i\in I_{n}}f_{\epsilon^{\prime}}\left(v_{i+\mathcal{N}},u_{i}\right)\right|$ which is computable by product and sum of computable functions, as the noise is computable. One can use this inequality to compute a modulus of continuity. 
\end{proof}

\subsection{Computability of closed sets}

We want to characterize the complexity of a closed set $K$ of a computable metric space $(X,\mathfrak{D}, d)$. The set is not necessarily discrete, so we can't have a direct definition of computable closed set. Formally, the computability of $K$ is equivalent to the ability of displaying it on a screen with arbitrary precision: is it possible to compute the set of \define{ideal balls} (open or closed) intersecting $K$, respectively defined by $B(z,r)=\{x\in X: d(x,z)<r\}$ and $\overline{B(z,r)}=\{x\in X: d(x,z)<r\}$with $z\in\mathcal{D}$ and $r\in\Q$. More precisely, we are interested in the complexity of the following sets:
\[
\mathcal{N}_{\textrm{out}}(K) := \left\{(x,r)\in\mathfrak{D}\times\Q:d(x,K)\leq r\right\} \quad\textrm{ and }\quad \mathcal{N}_{\textrm{in}}(K) := \left\{(x,r)\in\mathfrak{D}\times\Q:d(x,K)< r\right\}.
\]

The algorithmic complexity of these sets can be measured by considering the arithmetic hierarchies~\cite{Rogers-1987}. With this analogy, the set $K$ is said to be \define{$\Pi_n$-computable} if there exists a computable Boolean function $f:\mathfrak{D}\times\Q\times\N^k\to \{0,1\}$ such that:
\[
(x,r)\in\mathcal{N}_{\textrm{out}}(K) \Longleftrightarrow K\cap\overline{B(x,r)}\ne\emptyset\Longleftrightarrow \underset{\text{$k$ alternating quantifiers}}
{\underbrace{\forall y_1,\exists y_2,\forall y_3,\dots}}
f\left(x,r,y_1,\dots,y_k\right)=1  .
\]
Whereas $K$ is said to be \define{$\Sigma_n$-computable} if there exists a computable Boolean function $f:\mathfrak{D}\times\Q\times\N^k\to \{0,1\}$ such that:
\[
(x,r)\in\mathcal{N}_{\textrm{int}}(K) \Longleftrightarrow K\cap {B(x,r)}\ne\emptyset \Longleftrightarrow \underset{\text{$k$ alternating quantifiers}}
{\underbrace{ \exists y_1,\forall y_2,\exists y_3,\dots}}
f\left(x,r,y_1,\dots,y_k\right)=1  .
\]

A subset $K\subseteq X$ of a computable metric space $(X,\mathfrak{D},d)$ is \define{recursively compact} if $K$ is compact and it is semi-decidable whether a given finite set of open balls (encoded as elements of $\mathfrak{D}\times\Q$) covers $K$. For example, the computable metric spaces $[0,1]$, $\M(\A^\Z)$ and $[0,1]\times\M(\A^\Z)$ are known to be recursively compact, and a $\Pi_1$-computable closed subset of a recursively compact set is also recursively compact.

The definitions of $\Ml$ and $\Ms$ both use sequences of sets. To measure their complexity, we can use the complexity of the whole sequence at once: in our case we can use the notion of \define{uniform $\Pi_1$-computability}, where each set can be described by the same algorithm. The following equivalences for its definition are well-known, see for example \cite{HS18}:
\begin{prop}\label{prop:EquivalencePi1}
Let $\left(K_{i}\right)_{i\in\N}$ be a family of closed sets of a recursively compact set $(X,\mathfrak{D},d)$. The following propositions are equivalent, we say that the family is \define{uniformly $\Pi_1$-computable}.
\begin{enumerate}
	\item There exists $f:\N\times\mathfrak{D}\times\Q\times\N\to\left\{ 0,1\right\} $ computable such that  $\forall i\in\N,\forall(x,r)\in\mathfrak{D}\times\Q$,
	\[
	K_{i}\cap\overline{B}\left(x,r\right)\neq\emptyset\Longleftrightarrow\forall n\in\N,f\left(i,x,r,n\right)=1.
	\]
	\item There exists $g:\N\times\N\times\N\to\left\{ 0,1\right\} $ computable such that $\forall i\in\N,\forall j\in\N$,
	\[
	K_{i}\subset U_j\Longleftrightarrow\exists n\in\N,g\left(i,j,n\right)=1
	\]
	where $(U_j)_{j\in\N}$ is a computable enumeration of finite union of rational balls.
	\item There exists $k:\N\times\mathfrak{D}\times\Q\times\N\to\{0,1\} $ such that $\forall i\in\N$,
	\[
	K_{i}=X\backslash\bigcup_{(x,r)\in E_{i}}B(x,r) \textrm{ where }E_i=\{(x,r):\exists n\in\N,\ k(i,x,r,n)=1\}
	\]
	\item There exists $h:\N\times\M_{\sigma}\left(\A^{\Z}\right)\to\R^{+}$ computable such that $\forall i\in\N$, if $h_{i}$ denotes $h\left(i,\cdot\right)$, then
	\[
	K_{i}=h_{i}^{-1}\left(\left\{ 0\right\} \right).
	\]
\end{enumerate}
\end{prop}

\subsection{$\protect\M_{\left[\epsilon^{\prime},\epsilon^{\prime\prime}\right]}$
is $\Pi_{1}$-computable uniformly in $\epsilon^{\prime}$ and $\epsilon^{\prime\prime}$}

To characterize the complexity of $\Ml$, we need to precise the complexity that a closed ball intersects $\M_{\epsilon}$ for $\epsilon$ in some interval. That why one introduces the following set of measures and we want to precise its complexity uniformly according to its bounds. For $\epsilon', \epsilon''\in\Q_+$ such that $\epsilon'<\epsilon''$, define 
\[
\M_{\left[\epsilon',\epsilon''\right]}=\bigcup_{\epsilon'\leq\epsilon\leq\epsilon''}\M_{\epsilon}.
\]

The set $\M_{\left[\epsilon',\epsilon''\right]}$ is closed. Indeed, fix a sequence $\left(\pi_{n}\right)_{n\in\N}$ in $\M_{\left[\epsilon^{\prime},\epsilon^{\prime\prime}\right]}$ converging to $\pi\in\M(\A^\Z)$. For all $n$, there exists $\epsilon_{n}\in\left[\epsilon^{\prime},\epsilon^{\prime\prime}\right]$ such that $\pi_{n}\in\M_{\epsilon_{n}}$. By compactness, we can suppose that $\left(\epsilon_{n}\right)_{n\in\N}$ converges to some $\epsilon\in\left[\epsilon^{\prime},\epsilon^{\prime\prime}\right]$. By Lemma~\ref{Lemma-ContinuityLemma}, we deduce that $\pi\in\M_{\epsilon}\subset\M_{\left[\epsilon^{\prime},\epsilon^{\prime\prime}\right]}$.

\begin{prop}
If the perturbation is computable, then $\protect\M_{\left[\epsilon',\epsilon''\right]}$ is $\Pi_{1}$-computable uniformly in $\epsilon'$ and $\epsilon''$. In other words, there exists a computable map $\varphi:\Q\times\Q\times\mathfrak{P}\times\Q\times\N\to\{0,1\}$ such that 
\[
\overline{B(x,r)}\cap\M_{\left[\epsilon',\epsilon''\right]} \ne \emptyset \Longleftrightarrow \forall n\in\N,\ \varphi(\epsilon',\epsilon'',x,r,n)=1.
\]
\end{prop}

\begin{proof}
Define $\Psi: \left[0,1\right]\times\M_{\sigma}\left(\A^\Z\right)\longrightarrow\R^{+}$ as $\Psi:\left(\epsilon,\mu\right)\longmapsto d_{\M}\left(F_{\epsilon}\mu,\mu\right)$. By composition and Lemma~\ref{lemma.calculNoise} this function is computable, and so $\Psi^{-1}\left(0\right)$ is a closed $\Pi_{1}$-computable subset of $\left[0,1\right]\times\M_{\sigma}\left(\A^\Z\right)$. 

Finally we have for all $\epsilon^{\prime},\epsilon^{\prime\prime},x,r^{\prime},$
\[
\M_{\left[\epsilon^{\prime},\epsilon^{\prime\prime}\right]}\cap\overline{B}\left(x,r^{\prime}\right)\neq\emptyset\Longleftrightarrow\Psi^{-1}\left(0\right)\cap\left(\left[\epsilon^{\prime},\epsilon^{\prime\prime}\right]\times\overline{B}\left(x,r^{\prime}\right)\right)\neq\emptyset.
\]
As $\left[\epsilon^{\prime},\epsilon^{\prime\prime}\right]\times\overline{B}\left(x,r^{\prime}\right)=\left[\frac{\epsilon^{\prime}+\epsilon^{\prime\prime}}{2}\pm\frac{\epsilon^{\prime\prime}-\epsilon^{\prime}}{2}\right]\times\overline{B}\left(x,r^{\prime}\right)$ is a rational ball for the computable metric space $\left[0,1\right]\times\M(\A^\Z)$ with the distance $D\left(\left(s,\mu\right),\left(t,\nu\right)\right)=\max\left(\left|t-s\right|,d_{\M}\left(\mu,\nu\right)\right)$, and $\Psi^{-1}\left(0\right)$ is a closed $\Pi_{1}$-computable subset, we have our result.

\end{proof}

\subsection{Upper algorithmic bounds for $\Ml$ and $\Ms$}

We first give upper algorithmic bounds for $\Ml$ and $\Ms$ showing that there are $\Pi_{3}$-computable. Then we show that if $\Ml=\Ms$ or if $\Ml$ is uniformly approached, then their complexity are respectively $\Sigma_2$-computable and $\Pi_2$-computable.

\begin{prop}
\label{prop:computablePerturbation}
If $(F_\epsilon)_{\epsilon>0}$ is a computable perturbation of a CA, then $\Ml$ and $\Ms$ are $\Pi_{3}$-computable.
\end{prop}
\begin{proof}
Let $x\in\mathfrak{P},r\in\Q_+$. One has
\begin{eqnarray*}
 \overline{B(x,r)}\cap\Ml\ne\emptyset&\Longleftrightarrow& \forall r'>0,\forall\epsilon''>0,\exists\epsilon\leq\epsilon'',\ \overline{B(x,r+r')}\cap\M_{\epsilon}\ne\emptyset\\
&\Longleftrightarrow& \forall r'\in\Q^\ast_+,\forall\epsilon''\in\Q^\ast_+,\exists\epsilon'\in\Q^\ast_+\cap[0,\epsilon''],\ \overline{B(x,r+r')}\cap\M_{[\epsilon',\epsilon'']}\ne\emptyset\\
&\Longleftrightarrow&  \forall r'\in\Q^\ast_+,\forall\epsilon''\in\Q^\ast_+,\exists\epsilon'\in\Q^\ast_+\cap[0,\epsilon''],\forall n\in\N,\  \varphi(\epsilon',\epsilon'',x,r+r',n)=1
\end{eqnarray*}
thus $\Ml$ is $\Pi_{3}$-computable.  The first equivalence simply comes from the definition of $\Ml$.

In the same way, one has
\begin{eqnarray*}
 \overline{B(x,r)}\cap\Ms\ne\emptyset&\Longleftrightarrow& \forall r'>0,\exists\epsilon_0>0,\forall\epsilon<\epsilon_0,\ \overline{B(x,r+r')}\cap\M_{\epsilon}\ne\emptyset\\
&\overset{(\star)}{\Longleftrightarrow}& \forall r'\in\Q^\ast_+,\exists\epsilon_0\in\Q^\ast_+,\forall\epsilon',\epsilon''\in\Q^\ast_+\cap[0,\epsilon_0],\ \overline{B(x,r+r')}\cap\M_{[\epsilon',\epsilon'']}\ne\emptyset\\
&\Longleftrightarrow& \forall r'\in\Q^\ast_+,\exists\epsilon_0\in\Q^\ast_+,\forall\epsilon'\epsilon''\in\Q^\ast_+\cap[0,\epsilon_0],\forall n\in\N,\  \varphi(\epsilon',\epsilon'',x,r+r',n)=1
\end{eqnarray*}
thus $\Ms$ is $\Pi_{3}$-computable. The first equivalence simply comes from the definition of $\Ms$. The equivalence $(\star)$ is due to the following equality
\[\bigcap_{0<\epsilon<\epsilon_0}\M_{\epsilon}=\bigcap_{\underset{\epsilon',\epsilon''\in\Q}{0<\epsilon'<\epsilon''<\epsilon_0}}\M_{[\epsilon',\epsilon'']}.\]
The first inclusion is obvious. Reciprocally, let $\mu$ be a measure in the second set. For all $\epsilon$ one can construct a sequence of rationals $(\epsilon_n)_{n\in\N}$ which converges to $\epsilon$ and a sequence $(\pi_n)_{n\in\N}$ which converges to $\mu$ such that $\pi_n\in\M_{\epsilon_n}$. By the continuity lemma \ref{Lemma-ContinuityLemma}, we deduce that $\mu\in\M_{\epsilon}$.
\end{proof}

When $\Ml$ and $\Ms$ are equal, a better upper bound can be obtained.
\begin{prop}
If $(F_\epsilon)_{\epsilon>0}$ is a computable perturbation of a CA such that $\Ml=\Ms$, then this set is $\Sigma_2$-computable.
\end{prop}
\begin{proof}
Let $x\in\mathfrak{P},r\in\Q_+$. One has
\begin{eqnarray*}
B(x,r)\cap\Ms\ne\emptyset&\Longleftrightarrow& \exists r'>0,\exists\epsilon_0>0,\forall\epsilon\leq\epsilon_0,\  \overline{B(x,r-r')}\cap\M_{\epsilon}\ne\emptyset\\
	&\Longleftrightarrow& \exists r'\in\Q^\ast_+,\exists\epsilon_0\in\Q^\ast_+,\forall\epsilon',\epsilon''\in\Q^\ast_+\cap[0,\epsilon_0],\ \overline{B(x,r-r')}\cap\M_{[\epsilon',\epsilon'']}\ne\emptyset\\
	&\Longleftrightarrow& \exists r'\in\Q^\ast_+,\exists\epsilon_0\in\Q^\ast_+,\forall\epsilon',\epsilon''\in\Q^\ast_+\cap[0,\epsilon_0],\forall n\in\N,\  \varphi(\epsilon',\epsilon'',x,r-r',n)=1
\end{eqnarray*}
\end{proof}

Another natural constraint on the perturbation of the CA is when the zero-noise limit of invariant measures $\Ml$ is uniformly approached. We note that in this case, if $\Ml$ is not reduced to a singleton, $\Ms$ is empty.

\begin{prop} \label{prop:UniformApproachPi2}
If $(F_\epsilon)_{\epsilon>0}$ is a computable perturbation of a CA such that $\Ml$ is uniformly approached, then this set is $\Pi_2$-computable.
\end{prop}
\begin{proof}

Let $x\in\mathfrak{P},r\in\Q_+$, one has
\begin{eqnarray*}
	\overline{B(x,r)}\cap\Ml\ne\emptyset & \Longrightarrow & \forall\delta\in\Q_+^\ast, \forall r'\in\Q_+^\ast,\forall\epsilon_0\in\Q_+^\ast, \exists\epsilon\in[0,\epsilon_0]\cap\Q^\ast_+\cap D_\delta, \  \overline{B(x,r+r')}\cap\M_{\epsilon}\ne\emptyset  \\
	 & \Longrightarrow & \forall\delta\in\Q_+^\ast, \forall r'\in\Q_+^\ast, \forall\epsilon_0\in\Q_+^\ast, \exists\epsilon\in [0,\epsilon_0]\cap\Q^\ast_+, \exists y\in\mathfrak{P}, \ 
	\begin{cases}
		\M_{\epsilon}\subset B(y,\delta) \\
		\overline{B(x,r+r')}\cap\M_{\epsilon}\ne\emptyset
	\end{cases}  \\
	& \Longrightarrow & \forall\delta\in\Q_+^\ast, \forall r'\in\Q_+^\ast, \forall\epsilon_0\in\Q_+^\ast, \exists\epsilon\in [0,\epsilon_0]\cap\Q^\ast_+, \exists y\in\mathfrak{P}, \ 
	\begin{cases}
		\M_{\epsilon}\subset B(y,\delta)\\
		\overline{B(y,\delta)}\subset B(x,r+r'+2\delta)
	\end{cases} \\
	& \Longrightarrow & \underbrace{\forall\delta\in\Q_+^\ast, \forall r'\in\Q_+^\ast}, \forall\epsilon_0\in\Q_+^\ast, \exists\epsilon\in [0,\epsilon_0]\cap\Q^\ast_+, \ \M_{\epsilon}\subset B(x,r+ \underbrace{r'+2\delta}) \\
	& \Longrightarrow & \qquad \forall r''\in\Q_+^\ast, \qquad \forall\epsilon_0\in\Q_+^\ast, \exists\epsilon\in [0,\epsilon_0]\cap\Q^\ast_+, \ \M_{\epsilon}\subset B(x,r+\quad r'') \\
	& \Longrightarrow & \overline{B(x,r)}\cap\Ml\ne\emptyset
\end{eqnarray*}
where the first implication comes from Lemma~\ref{lemma-RestrictionDdelta} which states that $\Ml=\Acc{\epsilon\to 0}{(\M_{\epsilon})_{\epsilon\in D_{\delta}\cap\Q_+}}$. Thus we have the equivalence 
\[
\overline{B(x,r)}\cap\Ml\ne\emptyset \Longleftrightarrow \forall\delta\in\Q_+^\ast, \forall r'\in\Q_+^\ast, \forall\epsilon_0\in\Q_+^\ast, \exists\epsilon\in [0,\epsilon_0]\cap\Q^\ast_+, \exists y\in\mathfrak{P}, \ 
	\begin{cases}
		\M_{\epsilon}\subset B(y,\delta)\\
		\overline{B(y,\delta)}\subset B(x,r+r'+2\delta)
	\end{cases}
\]
and $\Ml$ is $\Pi_2$-computable since $\M_{\epsilon}$ and $\overline{B(y,\delta)}$ are $\Pi_1$-computable uniformly in $\epsilon$, $y$ and $\delta$ and it is semi-decidable to know if a $\Pi_1$-computable set is included in a rational ball.
\end{proof}

\section{Theorem of realization}\label{Section.TheoremeRealisation}

The purpose of this section is to prove the following theorem, showing that Proposition \ref{prop:UniformApproachPi2} can be thought as a characterization of $\Pi_2$-computable connected compact sets of measures.

\begin{thm}
\label{thm:realization}
Let $\mathcal{K}$ be a $\Pi_{2}$-computable connected compact subset of $\M\left(\A^\Z\right)$. There exists a cellular automaton $F$ on $\B^\Z$ with $\A\subset\B$ such that, perturbed by a uniform noise, we obtain $\Ml=\mathcal{K}$. Moreover, $\Ml$ is uniformly approached.
\end{thm}
\begin{rem}
 In this Theorem the cellular automaton cannot be assumed surjective since by Theorem 4.1 of \cite{MST19}, a surjective cellular automaton perturbed by a uniform noise has only the uniform Bernoulli measure as unique invariant measure. Thus $\Ml$ is the singleton with this measure.
\end{rem}

Fix a $\Pi_{2}$-computable connected set $\mathcal{K}$ of $\M\left(\A^\Z\right)$. Let us construct a cellular automaton $F$ on $\B^\Z\supset\A^\Z$ such that, perturbed by a uniform noise, we obtain $\Ml=\mathcal{K}\subset\M\left(\B^\Z\right)$ (more precisely, the natural projection of $\mathcal{K}$ on $\M\left(\B^\Z\right)$).

Following the characterizations of connected $\Pi_2$-computable set obtained in Proposition 6 of \cite{HS18} or Proposition 5 of \cite{GST23}, for a connected $\Pi_2$-computable set $\mathcal{K}$, there exists a  computable sequence of words $\left(\omega_{n}\right)$ be such that 
\[
\begin{cases}
d_\M\left(\widehat{\delta_{\omega_{n}}},\widehat{\delta_{\omega_{n+1}}}\right)\tendsto n{\infty}0\\
\text{Acc}\left(\left(\widehat{\delta_{\omega_{n}}}\right)_{n}\right)=\mathcal{K}
\end{cases}
\]
where $\widehat{\delta_{u}}$ is the shift-invariant measure supported on the periodic orbit of the repeated pattern $u\in\A^n$ for some $n\in\N$. As the distance between two successive measures tends to zero, we can write $\mathcal{K}=\text{Acc}\left(\text{Conv}\left(\widehat{\delta_{\omega_{n}}},\widehat{\delta_{\omega_{n+1}}}\right)\right)$.

\subsection{Main ideas}

The construction is essentially similar to the one described in~\cite{HS18}. It is in some ways different for two reasons: as we make computations on the invariant measures there is no \emph{initial} configuration, so we can consider events arbitrarily far in the past; also, as we take the limit $\epsilon\to 0$, there is no need for a merging process.  The main ideas are as follows:
\begin{itemize}
    \item A special symbol denoted by $*$ constructs three zones delimited by walls. The walls erase (almost) anything in their path and protect the content of their respective zones from outside. The $*$ can only appear with an error.
    \item The center zone is the \define{display zone}: delimited by a protective wall with speed $1$, it is where the successive $\omega_i$ will be sent. Without errors, only symbols of $\A$ are in it, and the CA only acts as the left shift $\sigma$ on those symbols. (See the introduction of section \ref{subsec:computations} and \ref{subsec:Errors})
    \item The right zone is the \define{computation zone}: delimited by a wall with logarithmic speed (using an odometer), it uses Turing machines to computes both the $\omega_i$ and the times $T_i$ at which it should begin displaying them. (See section \ref{subsec:computations})
     \item The left zone is the \define{format zone}: it is also delimited by a wall using a odometer, but is shifted to the left by 1 each step. The odometer is a time counter which is used to compare the age of the construction (the number of steps since the last $*$ symbol) to the one of a computing zone it can encounter. (See section \ref{subsec:Collisions})
    \item When two constructions collide, only the younger one survives. In the rare case where they have the same age, only the left one survives (this choice is arbitrary). (See section \ref{subsec:Collisions})
    \item The alphabet $\B$ is composed of $\A$ and all the symbols used for the construction (e.g. walls and signals).
\end{itemize}

\begin{figure}[h]
    \centering
    \begin{minipage}[c]{0.60\linewidth}
    \def\svgwidth{\columnwidth}
    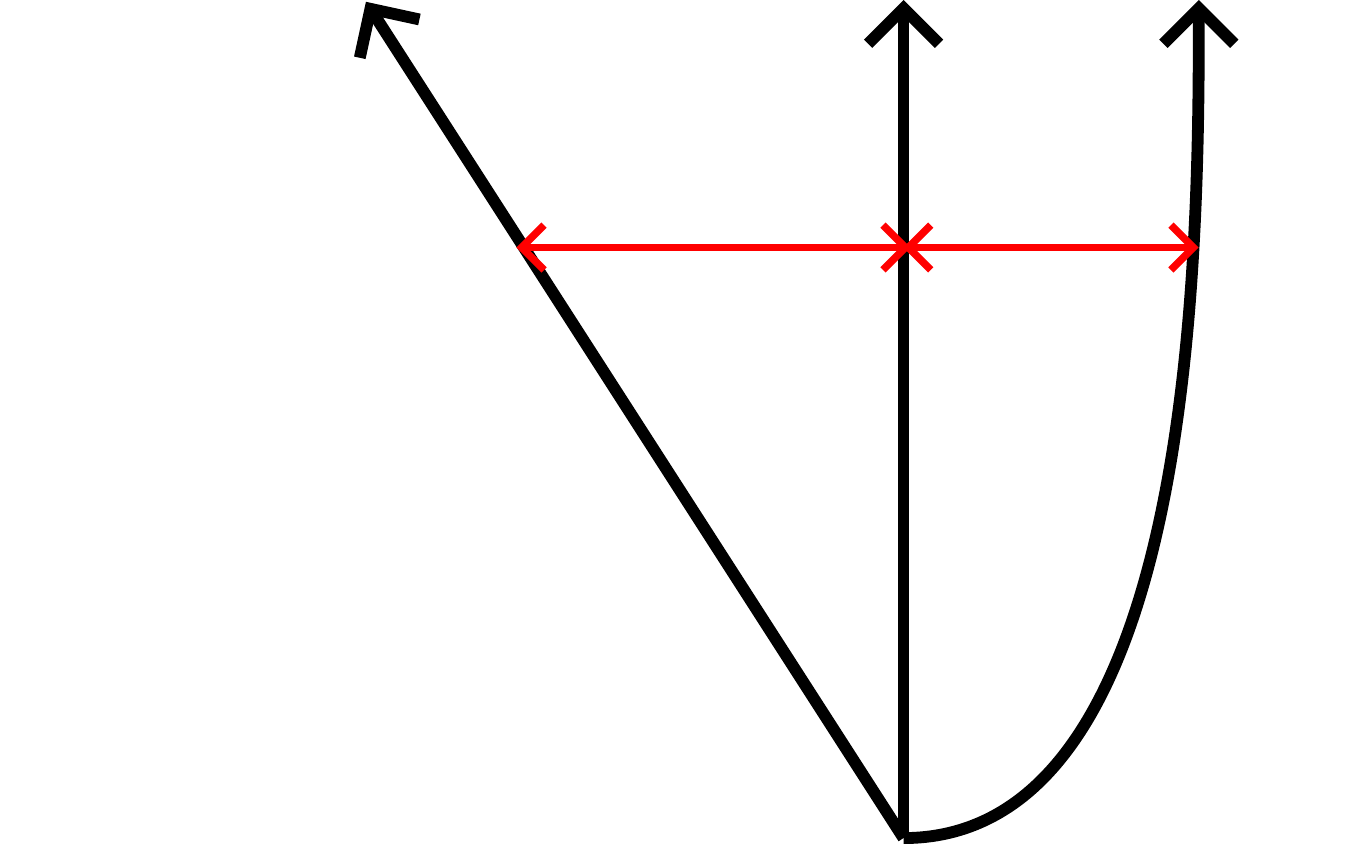
    \end{minipage}
    \begin{minipage}[c]{0.39\linewidth}
    \def\svgwidth{\columnwidth}
    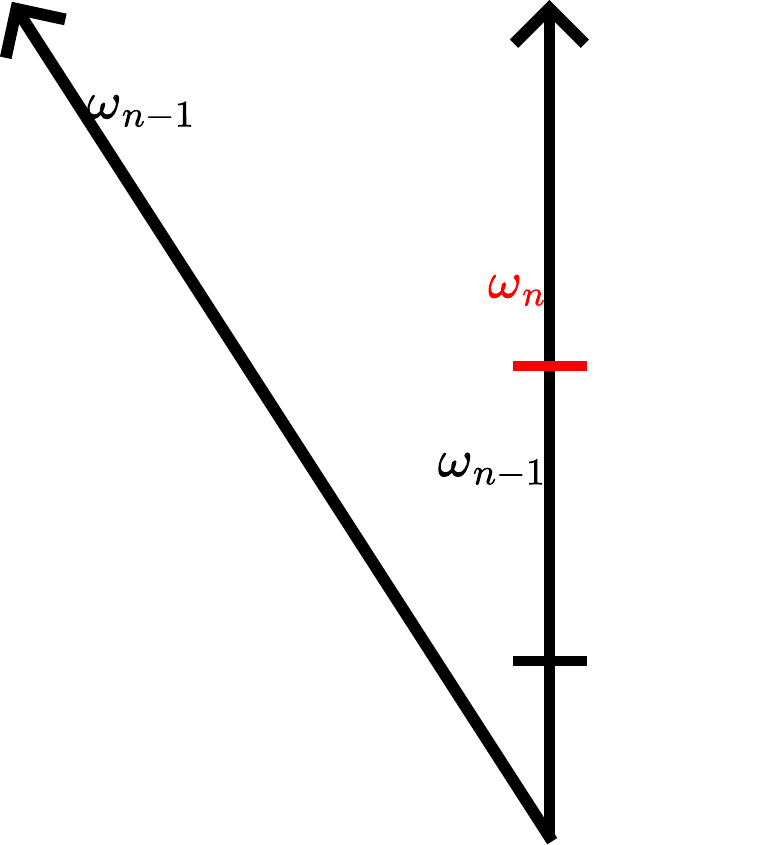
    \end{minipage}
    \caption{Scheme of a space-time diagram the construction (with time going up). Left: a $*$ symbol creates walls that separates into three zones. After $t$ iterations, the central display zone is of linear size, while the other are of logarithmic size. Right: in the display zone, the word $\omega_n$ begins to be sent after $T_n$ iterations. All symbols from $\A$ are shifted along with the left wall.}
\end{figure}

For a random configuration following a given $F_\epsilon$-invariant measure, such a construction ``lives" for an average of $\frac{1}{2}\sqrt{\frac{|\B|}{\epsilon}}$, before being replaced by a new one as explained in Subsection~\ref{section:choseTn}. Thus $t$ iterations after a $*$ symbol, the display zone is of size $t$, while the computation zone is of size $2\log(t)$. In particular, the size of the computation zone is negligible in front of the size of the display zone. Moreover, with a good choice of $T_n$, the display zone will be mostly composed of the words $\omega_n$ and $\omega_{n+1}$.  

Therefore as $\epsilon$ goes to $0$, the invariant measures will be closer and closer to a convex combination of $\widehat{\delta_{\omega_{n}}}$ and $\widehat{\delta_{\omega_{n+1}}}$. Moreover, as $d\left(\widehat{\delta_{\omega_{n}}},\widehat{\delta_{\omega_{n+1}}}\right)\tendsto n{\infty}0$, it is equivalent to be close to $\widehat{\delta_{\omega_{n}}}$ (with larger and larger $n$), and thus have the same accumulation points. Note that this last condition implies that the set of accumulation points is connected; but by Proposition \ref{prop.MlConnected} $\Ml$ always is, so this is not a restriction.

\subsection{Action of the CA in the computation zone}\label{subsec:computations}

In the display zone, only symbols from $\A$ are created. The CA acts as the shift $\sigma$ on them: each iteration, they are moved 1 step to the left, along with the wall between the format zone and the display zone. The new symbol on the rightmost cell of the display zone is taken from the leftmost cell of the computation zone (see section \ref{subsec:copyingOmega}), through the wall symbol denoted by $W$.

In this section, we describe how the CA acts on the computation zone: how it computes the sequence of words $\omega_n$ and feeds them to the display zone, how the times $T_n$ at which they are sent are computed and why they were chosen.

\subsubsection{The choice of the times $T_n$}\label{section:choseTn}

For all integer $n$, $T_n$ denotes the time elapsed since a $*$ symbol at which the computation zone begins sending the word $\omega_n$ to the display zone. In this section, we describe how we can choose them.

Suppose a realization of a infinite space-time diagram is given. We want to study $Y$, the number of steps that have been done since a $*$ symbol appeared on the right side of the dependence cone of the cell at position $0$. We can naturally order the cells on this cone as in Figure \ref{fig:timeLastStar} (the cell at position $(i, -t)$ in space-time is assigned the place $i + \frac{t(t-1)}{2}$ for $1 \leq i \leq t$) and denote by $X$ the numbering of the first cell where a $*$ symbol occurs. By construction, $Y = |t| = \left\lceil \frac{-1+\sqrt{1+8X}}{2}\right\rceil$ ($Y = \lceil g^{-1}(X) \rceil$ for $g(x)\coloneqq\frac{x(x+1)}{2}$). Moreover, since the errors are independent and a $*$ symbol can only appear via an error, with a probability of $\frac{\epsilon}{|\B|}$ to appear on each cell, we have $X\sim\mathcal{G}(\frac{\epsilon}{|\B|})$. For clarity, we define $\boldepsilon = \frac{\epsilon}{|\B|}$ so that $X\sim\mathcal{G}(\boldepsilon)$.

\begin{figure}[h]
    \centering
    \includegraphics[scale=1]{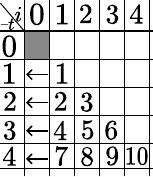}
    \caption{Ordering of the considered cells ($i>0$ and $-i \leq t < 0$), from left to right and top to bottom. The arrow indicates for each the number of steps $Y$ for a $*$ symbol at this row to reach the grayed cell in position $(0,0)$ in space-time.}
    \label{fig:timeLastStar}
\end{figure}

Note that for a given $I_\epsilon = [a_\epsilon; b_\epsilon]$ with $a_{\epsilon}^{2}=o\left(\frac{1}{\epsilon}\right)$ and $\frac{1}{\epsilon}=o\left(b_{\epsilon}^{2}\right)$, we have $P\left(Y\in I_{\epsilon}\right)\tendsto{\epsilon}{0}1$. Indeed,
\begin{align*}
    P\left(Y\in I_{\epsilon}\right) &= P\left(X\in[g(a_\epsilon); g(b_\epsilon)]\right) \\
    &= 1-P\left(X<g(a_\epsilon)\right) - P\left(X>g(b_\epsilon)\right) \\
    &= \left(1-\boldepsilon\right)^{g(a_\epsilon)} - \left(1-\boldepsilon\right)^{g(b_\epsilon)} \\
    &= \exp\left(g(a_\epsilon)\ln\left(1-\frac{\epsilon}{|\B|}\right) \right) - \exp\left(g(b_\epsilon)\ln\left(1-\frac{\epsilon}{|\B|}\right) \right)
\end{align*}
and $g(a_\epsilon)\ln(1-\frac{\epsilon}{|\B|})\sim -a_\epsilon^2\frac{\epsilon}{2|\B|}\tendsto{\epsilon}{0}0$ and $g(b_\epsilon)\ln(1-\frac{\epsilon}{|\B|})\tendsto{\epsilon}{0}-\infty$.

For a fixed constant $0<\eta\ll 1$, $a_{\epsilon}=\left(\frac{1}{\epsilon}\right)^{\frac{1}{2}-\eta}$ and $b_{\epsilon}=\left(\frac{1}{\epsilon}\right)^{\frac{1}{2}+\eta}$ verify this property. We can then define $T_{n}=2^{2^{n}}$ and $n_\epsilon \coloneqq \max\{n\in\N \mid T_n \leq a_{\epsilon} \}$. This choice is motivated by the following properties:
\begin{enumerate}
\item It is easily computable by a Turing machine, using the relation $T_{n+1}=T_{n}^{2}$; the space and time complexity are not an issue here (we have $T_{n+1} - T_{n}$ time steps and a space of size $2\log T_{n}$).
\item $\frac{T_{n}-T_{n-1}}{T_{n}}\tendsto n{\infty}1$: the size of the displayed zone computed in the first $n-1$ steps is negligible compared to one that will be displayed during the $n^{\text{th}}$ one.
\item At fixed $\epsilon>0$, one has $I_{\epsilon}\subset\left[T_{n_{\epsilon}};T_{n_{\epsilon}+2}\right]$. Indeed, we have $T_{n_{\epsilon}+1}>a_{\epsilon}$, and thus 
\[
b_{\epsilon}<T_{n_{\epsilon}+1}^{\frac{1/2+\eta}{1/2-\eta}}\leq T_{n_{\epsilon}+1}^{2}=T_{n_{\epsilon}+2}
\]
(we fix $\eta < \frac{1}{6}$ so that the middle inequality is true).
\end{enumerate}

We can suppose each term $\omega_n$ to be computable in time $T_n - T_{n-1}$ and in space $\log(T_{n-1})$. If it is not the case, we can just repeat the last word computed $\omega_{n-1}$ and try to compute $\omega_n$ again. In the end, we still have the same accumulation points of $\widehat{\delta_{\omega_n}}$, which are $\mathcal{K}$.

Therefore, for a fixed $\epsilon$, the last $*$ symbol occurred $t$ steps before, with $t\in I_{\epsilon}$ with high probability. Moreover, as  $I_{\epsilon}\subset\left[T_{n_{\epsilon}};T_{n_{\epsilon}+2}\right]$ and  $\frac{T_{n}-T_{n-1}}{T_{n}}\tendsto n{\infty}1$, it means that with high probability, we only see sub-words of $\omega_{n_{\epsilon}}$ and $\omega_{n_{\epsilon}+1}$ (more details in section \ref{subsec:calculs}). Thus any $F_\epsilon$-invariant probability $\pi_\epsilon$ will be close (for the distance $d_\M$) to measures with support on the periodic orbits associated to those words. $\pi_{\epsilon}$ will be closer and closer to $\text{Conv}\left(\widehat{\delta_{\omega_{n_{\epsilon}}}},\widehat{\delta_{\omega_{n_{\epsilon}+1}}}\right)$. More details in section \ref{subsec:conclusionPreuve}.

\subsubsection{Computing $T_n$} \hfill\\

We describe a system of counters and signals designed such that when (and only when) $T_{n}=2^{2^{n}}$ (for any $n\in\N^{*}$) steps have passed after a $*$ symbol, the cell at the right of the wall separating the display zone and the computation zone is in a specific state. It allows the computation to know when to switch from displaying $\omega_{n-1}$ to $\omega_{n}$. Apart from this result, the other properties will not be used in the other sections.

We consider three layers, each one with a different alphabet:
\begin{itemize}
\item $\A_{\text{count}}=\left\{ 0,1,2,\#\right\} $
\item $\A_{\text{signal}}=\left\{ S_{1},S_{2},\#\right\} $
\item $\A_{\text{mines}}=\left\{ T,M,C,\#\right\} $
\end{itemize}
where $\#$ is a neutral symbol. Let us describe the action of the CA on each layer.

On the first layer the CA acts by iterating an odometer in base 2 with a carry that propagates (in our case, to the right). One can see a precise description in Section 4.2.2 of \cite{HS18}, and the first steps are illustrated in Figure \ref{fig:compteurs}. The first cell, next to the wall, is added 1 each step, and thus is alternating between ``$1$" and ``$2$". Each time a carry ``$2$" reaches the last numbered cell, it creates a new cell $\left(1,S_{1},\#\right)$ to the right.

The second layer consists of signals that propagates at speed 1 towards the center, but only if the two consecutive cells left to them forms the pattern ``$20$" (formally, it propagates through the ``$1$"s, but see the carry in advance). A signal $S_{1}$ becomes a $S_{2}$ if it encounters a mine $M$. When a signal arrives close to the center (the cells left to it are $W1$), it disappears and creates left to it :
\begin{itemize}
\item a clearer $C$ if it is a signal $S_{2}$.
\item a trapper $T$ if it is a signal $S_{1}$. This event is the one marking a step $T_{n}=2^{2^{n}}$.
\end{itemize}
Finally, the third layer consists of:
\begin{itemize}
\item a trapper signal $T$ that follows the carry ``$2$"  to the right, leaving a mine $M$ on its path (except on the cell next to the wall). It disappears when the carry creates a new digit. 
\item mines $M$ created by a $T$ signal, that stay in place until a clearer arrives. 
\item a clearer signal $C$ that also follows the carry, but disappears when encountering a mine. 
\end{itemize}
Figure \ref{fig:compteurs} illustrates the first 16 steps on the computation zone after a $*$ symbol.

\begin{figure}

\begin{centering}
\begin{tabular}{|c|c|c|c|c|c|c|c|c|c|c|c|}
\hline 
\multirow{2}{*}{$t=16$} & \multirow{2}{*}{$W$} & \multicolumn{2}{c|}{$\mathbf{2}$} & \multicolumn{2}{c|}{$\mathbf{1}$} & \multicolumn{2}{c|}{$\mathbf{1}$} & \multicolumn{2}{c|}{$\mathbf{1}$} & \multicolumn{2}{c|}{}\tabularnewline
\cline{3-12} \cline{4-12} \cline{5-12} \cline{6-12} \cline{7-12} \cline{8-12} \cline{9-12} \cline{10-12} \cline{11-12} \cline{12-12} 
 &  &  & \textcolor{purple}{$T$} &  &  &  &  &  &  &  & \tabularnewline
\hline 
\multirow{2}{*}{$t=15$} & \multirow{2}{*}{$W$} & \multicolumn{2}{c|}{$\mathbf{1}$} & \multicolumn{2}{c|}{$\mathbf{1}$} & \multicolumn{2}{c|}{$\mathbf{1}$} & \multicolumn{2}{c|}{$\mathbf{1}$} & \multicolumn{2}{c|}{}\tabularnewline
\cline{3-12} \cline{4-12} \cline{5-12} \cline{6-12} \cline{7-12} \cline{8-12} \cline{9-12} \cline{10-12} \cline{11-12} \cline{12-12} 
 &  &  &  & \textcolor{red}{$S_{1}$} &  &  &  &  &  &  & \tabularnewline
\hline 
\multirow{2}{*}{$t=14$} & \multirow{2}{*}{$W$} & \multicolumn{2}{c|}{$\mathbf{2}$} & \multicolumn{2}{c|}{$\mathbf{0}$} & \multicolumn{2}{c|}{$\mathbf{1}$} & \multicolumn{2}{c|}{$\mathbf{1}$} & \multicolumn{2}{c|}{}\tabularnewline
\cline{3-12} \cline{4-12} \cline{5-12} \cline{6-12} \cline{7-12} \cline{8-12} \cline{9-12} \cline{10-12} \cline{11-12} \cline{12-12} 
 &  &  &  &  &  & \textcolor{red}{$S_{1}$} &  &  &  &  & \tabularnewline
\hline 
\multirow{2}{*}{$t=13$} & \multirow{2}{*}{$W$} & \multicolumn{2}{c|}{$\mathbf{1}$} & \multicolumn{2}{c|}{$\mathbf{2}$} & \multicolumn{2}{c|}{$\mathbf{0}$} & \multicolumn{2}{c|}{$\mathbf{1}$} & \multicolumn{2}{c|}{}\tabularnewline
\cline{3-12} \cline{4-12} \cline{5-12} \cline{6-12} \cline{7-12} \cline{8-12} \cline{9-12} \cline{10-12} \cline{11-12} \cline{12-12} 
 &  &  &  &  &  &  &  & \textcolor{red}{$S_{1}$} &  &  & \tabularnewline
\hline 
\multirow{2}{*}{$t=12$} & \multirow{2}{*}{$W$} & \multicolumn{2}{c|}{$\mathbf{2}$} & \multicolumn{2}{c|}{$\mathbf{1}$} & \multicolumn{2}{c|}{$\mathbf{0}$} & \multicolumn{2}{c|}{$\mathbf{1}$} & \multicolumn{2}{c|}{}\tabularnewline
\cline{3-12} \cline{4-12} \cline{5-12} \cline{6-12} \cline{7-12} \cline{8-12} \cline{9-12} \cline{10-12} \cline{11-12} \cline{12-12} 
 &  &  &  &  &  &  &  & \textcolor{red}{$S_{1}$} &  &  & \tabularnewline
\hline 
\multirow{2}{*}{$t=11$} & \multirow{2}{*}{$W$} & \multicolumn{2}{c|}{$\mathbf{1}$} & \multicolumn{2}{c|}{$\mathbf{1}$} & \multicolumn{2}{c|}{$\mathbf{0}$} & \multicolumn{2}{c|}{$\mathbf{1}$} & \multicolumn{2}{c|}{}\tabularnewline
\cline{3-12} \cline{4-12} \cline{5-12} \cline{6-12} \cline{7-12} \cline{8-12} \cline{9-12} \cline{10-12} \cline{11-12} \cline{12-12} 
 &  &  &  &  &  &  &  & \textcolor{red}{$S_{1}$} &  &  & \tabularnewline
\hline 
\multirow{2}{*}{$t=10$} & \multirow{2}{*}{$W$} & \multicolumn{2}{c|}{$\mathbf{2}$} & \multicolumn{2}{c|}{$\mathbf{0}$} & \multicolumn{2}{c|}{$\mathbf{2}$} & \multicolumn{2}{c|}{} & \multicolumn{2}{c|}{}\tabularnewline
\cline{3-12} \cline{4-12} \cline{5-12} \cline{6-12} \cline{7-12} \cline{8-12} \cline{9-12} \cline{10-12} \cline{11-12} \cline{12-12} 
 &  &  &  &  &  &  &  &  &  &  & \tabularnewline
\hline 
\multirow{2}{*}{$t=9$} & \multirow{2}{*}{$W$} & \multicolumn{2}{c|}{$\mathbf{1}$} & \multicolumn{2}{c|}{$\mathbf{2}$} & \multicolumn{2}{c|}{$\mathbf{1}$} & \multicolumn{2}{c|}{} & \multicolumn{2}{c|}{}\tabularnewline
\cline{3-12} \cline{4-12} \cline{5-12} \cline{6-12} \cline{7-12} \cline{8-12} \cline{9-12} \cline{10-12} \cline{11-12} \cline{12-12} 
 &  &  &  &  &  &  &  &  &  &  & \tabularnewline
\hline 
\multirow{2}{*}{$t=8$} & \multirow{2}{*}{$W$} & \multicolumn{2}{c|}{$\mathbf{2}$} & \multicolumn{2}{c|}{$\mathbf{1}$} & \multicolumn{2}{c|}{$\mathbf{1}$} & \multicolumn{2}{c|}{} & \multicolumn{2}{c|}{}\tabularnewline
\cline{3-12} \cline{4-12} \cline{5-12} \cline{6-12} \cline{7-12} \cline{8-12} \cline{9-12} \cline{10-12} \cline{11-12} \cline{12-12} 
 &  &  & \textcolor{blue}{$C$} &  & \textcolor{brown}{$M$} &  &  &  &  &  & \tabularnewline
\hline 
\multirow{2}{*}{$t=7$} & \multirow{2}{*}{$W$} & \multicolumn{2}{c|}{$\mathbf{1}$} & \multicolumn{2}{c|}{$\mathbf{1}$} & \multicolumn{2}{c|}{$\mathbf{1}$} & \multicolumn{2}{c|}{} & \multicolumn{2}{c|}{}\tabularnewline
\cline{3-12} \cline{4-12} \cline{5-12} \cline{6-12} \cline{7-12} \cline{8-12} \cline{9-12} \cline{10-12} \cline{11-12} \cline{12-12} 
 &  &  &  & \textcolor{brown}{$S_{2}$} & \textcolor{brown}{$M$} &  &  &  &  &  & \tabularnewline
\hline 
\multirow{2}{*}{$t=6$} & \multirow{2}{*}{$W$} & \multicolumn{2}{c|}{$\mathbf{2}$} & \multicolumn{2}{c|}{$\mathbf{0}$} & \multicolumn{2}{c|}{$\mathbf{1}$} & \multicolumn{2}{c|}{} & \multicolumn{2}{c|}{}\tabularnewline
\cline{3-12} \cline{4-12} \cline{5-12} \cline{6-12} \cline{7-12} \cline{8-12} \cline{9-12} \cline{10-12} \cline{11-12} \cline{12-12} 
 &  &  &  &  & \textcolor{brown}{$M$} & \textcolor{red}{$S_{1}$} &  &  &  &  & \tabularnewline
\hline 
\multirow{2}{*}{$t=5$} & \multirow{2}{*}{$W$} & \multicolumn{2}{c|}{$\mathbf{1}$} & \multicolumn{2}{c|}{$\mathbf{2}$} & \multicolumn{2}{c|}{} & \multicolumn{2}{c|}{} & \multicolumn{2}{c|}{}\tabularnewline
\cline{3-12} \cline{4-12} \cline{5-12} \cline{6-12} \cline{7-12} \cline{8-12} \cline{9-12} \cline{10-12} \cline{11-12} \cline{12-12} 
 &  &  &  &  & \textcolor{purple}{$T$} &  &  &  &  &  & \tabularnewline
\hline 
\multirow{2}{*}{$t=4$} & \multirow{2}{*}{$W$} & \multicolumn{2}{c|}{$\mathbf{2}$} & \multicolumn{2}{c|}{$\mathbf{1}$} & \multicolumn{2}{c|}{} & \multicolumn{2}{c|}{} & \multicolumn{2}{c|}{}\tabularnewline
\cline{3-12} \cline{4-12} \cline{5-12} \cline{6-12} \cline{7-12} \cline{8-12} \cline{9-12} \cline{10-12} \cline{11-12} \cline{12-12} 
 &  &  & \textcolor{purple}{$T$} &  &  &  &  &  &  &  & \tabularnewline
\hline 
\multirow{2}{*}{$t=3$} & \multirow{2}{*}{$W$} & \multicolumn{2}{c|}{$\mathbf{1}$} & \multicolumn{2}{c|}{$\mathbf{1}$} & \multicolumn{2}{c|}{} & \multicolumn{2}{c|}{} & \multicolumn{2}{c|}{}\tabularnewline
\cline{3-12} \cline{4-12} \cline{5-12} \cline{6-12} \cline{7-12} \cline{8-12} \cline{9-12} \cline{10-12} \cline{11-12} \cline{12-12} 
 &  &  &  & \textcolor{red}{$S_{1}$} &  &  &  &  &  &  & \tabularnewline
\hline 
\multirow{2}{*}{$t=2$} & \multirow{2}{*}{$W$} & \multicolumn{2}{c|}{$\mathbf{2}$} & \multicolumn{2}{c|}{} & \multicolumn{2}{c|}{} & \multicolumn{2}{c|}{} & \multicolumn{2}{c|}{}\tabularnewline
\cline{3-12} \cline{4-12} \cline{5-12} \cline{6-12} \cline{7-12} \cline{8-12} \cline{9-12} \cline{10-12} \cline{11-12} \cline{12-12} 
 &  &  &  &  &  &  &  &  &  &  & \tabularnewline
\hline 
\multirow{2}{*}{$t=1$} & \multirow{2}{*}{$W$} & \multicolumn{2}{c|}{$\mathbf{1}$} & \multicolumn{2}{c|}{$\#$} & \multicolumn{2}{c|}{$\#$} & \multicolumn{2}{c|}{$\#$} & \multicolumn{2}{c|}{$\A_{\text{count}}$}\tabularnewline
\cline{3-12} \cline{4-12} \cline{5-12} \cline{6-12} \cline{7-12} \cline{8-12} \cline{9-12} \cline{10-12} \cline{11-12} \cline{12-12} 
 &  & $\#$ & $\#$ & $\#$ & $\#$ & $\#$ & $\#$ & $\#$ & $\#$ & $\A_{\text{signal}}$ & $\A_{\text{mines}}$\tabularnewline
\hline 
\end{tabular}\caption{The first 16 steps. The $\#$ symbols are omitted after the first step. A trapper signal $T$ is produced on $T_1=4$ and $T_2=16$.}\label{fig:compteurs}
\par\end{centering}
\end{figure}

Let us first observe the following: in the counter, there cannot be two successive ``$2$" symbols. Indeed, if such a configuration existed, by going back in time there would have been a configuration of the form \begin{tabular}{|c|c|c|c|c|}
\hline 
$W$ & $2$ & $2$ & $\omega_{1}$ & $\cdots$\tabularnewline
\hline 
\end{tabular}, with $\omega\in\A_{\text{count}}^{*}$, but such a configuration has no predecessor.

Let us prove by induction the following property $\mathcal{P}\left(n\right)$ for $n\in\N^{*}$: between $T_{n}$ and $T_{n+1}-1$, a trapper $T$ is created only at $T_{n}$ , and at $T_{n+1}-1$ steps, the counter is of the form \newline
\begin{tabular}{|c|c|c|c|c|c|c|c|c|c|c|c|}
\hline 
\multirow{2}{*}{$W$} & \multicolumn{2}{c|}{$1$} & \multicolumn{2}{c|}{$1$} & \multicolumn{2}{c|}{$1$} &  & \multicolumn{2}{c|}{$1$} & \multicolumn{2}{c|}{$1$}\tabularnewline
\cline{2-12} \cline{3-12} \cline{4-12} \cline{5-12} \cline{6-12} \cline{7-12} \cline{8-12} \cline{9-12} \cline{10-12} \cline{11-12} \cline{12-12} 
 & $\#$ & $\#$ & $S_{1}$ & $\#$ & $\#$ & $\#$ & $\cdots$ & $\#$ & $\#$ & $\#$ & $\#$\tabularnewline
\hline 
\end{tabular}
with $2^{n}$ ones.

We know that $\mathcal{P}\left(1\right)$ is true (see Figure \ref{fig:compteurs}). Suppose that it is true for $n-1$, and let us prove $\mathcal{P}\left(n\right)$.
\begin{enumerate}
\item By making a step on the local configuration, we know that a trapper $T$ is created on the first cell at $T_{n}$, and the counter is
\begin{tabular}{|c|c|c|c|c|c|}
\hline 
$W$ & $2$ & $1$ & $\cdots$ & $1$ & $\#$\tabularnewline
\hline 
\end{tabular}.
\item The $T$ (along with the carry) will propagate all the way to the right and leave a mine $m$ where there were ``$1$"s before. Because there is no carry immediately after, the configuration is now of the form
\begin{center}
\begin{tabular}{|c|c|c|c|c|c|c|c|c|c|c|c|c|}
\hline 
\multirow{2}{*}{$W$} & \multicolumn{2}{c|}{$u$} & \multicolumn{2}{c|}{$\omega_{1}$} & $\cdots$ & \multicolumn{2}{c|}{$\omega_{2^{n}-2}$} & \multicolumn{2}{c|}{$0$} & \multicolumn{2}{c|}{$1$} & $\#$\tabularnewline
\cline{2-13} \cline{3-13} \cline{4-13} \cline{5-13} \cline{6-13} \cline{7-13} \cline{8-13} \cline{9-13} \cline{10-13} \cline{11-13} \cline{12-13} \cline{13-13} 
 & $\#$ & $\#$ & $\#$ & $C$ & $\cdots$ & $\#$ & $C$ & $\#$ & $C$ & $S_{1}$ & $\#$ & \tabularnewline
\hline 
\end{tabular} 
\par\end{center}
with $u\in\left\{ 1,2\right\} $, $\omega=\left(\omega_{i}\right)\in\left\{ 0,1,2\right\} ^{2^{n}-2}$, a signal $S_{1}$ on the final ``$1$" and mines on all the cells of $\omega$ and the final ``$0$". There is $2^{n}-1$ mines in the counter. 
\item By a direct induction using the first lemma, there can be only the symbol ``$1$" on the first layer right of a $S$ signal and a ``$0$" immediately on its left, unless it is next to the wall. 
\item With this property and the first lemma, a signal $S$ arrives to the center on the configuration 
\begin{center}
\begin{tabular}{|c|c|c|c|c|c|c|c|}
\hline 
\multirow{2}{*}{$W$} & $1$ & $2$ & $0$ & $1$ & $1$ & $\cdots$ & $1$\tabularnewline
\cline{2-8} \cline{3-8} \cline{4-8} \cline{5-8} \cline{6-8} \cline{7-8} \cline{8-8} 
 & $\#$ & $\#$ & $\#$ & $S_{i}$ & $\#$ & $\cdots$ & $\#$\tabularnewline
\hline 
\end{tabular} 
\par\end{center}
then
\begin{center}
\begin{tabular}{|c|c|c|c|c|c|c|c|}
\hline 
\multirow{2}{*}{$W$} & $2$ & $0$ & $1$ & $1$ & $1$ & $\cdots$ & $1$\tabularnewline
\cline{2-8} \cline{3-8} \cline{4-8} \cline{5-8} \cline{6-8} \cline{7-8} \cline{8-8} 
 & $\#$ & $\#$ & $S_{i}$ & $\#$ & $\#$ & $\cdots$ & $\#$\tabularnewline
\hline 
\end{tabular} 
\par\end{center}
so finally a signal $S$ creates a $C$ or a $T$ signal after the next step, when the counter is of the form
\begin{center}
\begin{tabular}{|c|c|c|c|c|c|c|c|c|c|c|c|}
\hline 
\multirow{2}{*}{$W$} & \multicolumn{2}{c|}{$1$} & \multicolumn{2}{c|}{$1$} & \multicolumn{2}{c|}{$1$} & $\cdots$ & \multicolumn{2}{c|}{$1$} & \multicolumn{2}{c|}{$1$}\tabularnewline
\cline{2-12} \cline{3-12} \cline{4-12} \cline{5-12} \cline{6-12} \cline{7-12} \cline{8-12} \cline{9-12} \cline{10-12} \cline{11-12} \cline{12-12} 
 & $\#$ & $\#$ & $S_{i}$ & ? & $\#$ & ? & $\cdots$ & $\#$ & ? & $\#$ & $\#$\tabularnewline
\hline 
\end{tabular} 
\par\end{center}
If we denote by $k$ the number of ``$1$", this configuration occurs at time $\sum_{j=0}^{k-1}1\cdot2^{j}=2^{k}-1$, and the new $C$ or $T$ signal is created after $2^{k}$ steps. In particular, there is one signal $S$ arriving per power of $2$. 
\item Because there were initially $2^{n}-1$ mines in the counter, all the signals $S$ arriving at $2^{2^{n}+1}$, $2^{2^{n}+2}$, ..., $2^{2^{n}+2^{n}-1}$ encounter a mine, so they are $S_{2}$. They each create one $C$ signal, that clears one $M$ mine (before the next $S$ signal is created). Thus, the path is mine-free for the next signal arriving at $T_{n+1}$, and thus the configuration at $T_{n+1}-1$ is 
\begin{center}
\begin{tabular}{|c|c|c|c|c|c|c|c|c|c|c|c|}
\hline 
\multirow{2}{*}{$W$} & \multicolumn{2}{c|}{$1$} & \multicolumn{2}{c|}{$1$} & \multicolumn{2}{c|}{$1$} &  & \multicolumn{2}{c|}{$1$} & \multicolumn{2}{c|}{$1$}\tabularnewline
\cline{2-12} \cline{3-12} \cline{4-12} \cline{5-12} \cline{6-12} \cline{7-12} \cline{8-12} \cline{9-12} \cline{10-12} \cline{11-12} \cline{12-12} 
 & $\#$ & $\#$ & $S_{1}$ & $\#$ & $\#$ & $\#$ & $\cdots$ & $\#$ & $\#$ & $\#$ & $\#$\tabularnewline
\hline 
\end{tabular}
\par\end{center}
with $2^{n}$ symbols ``$1$".
\end{enumerate}

\subsubsection{Computing of $\omega_n$ at time $T_{n-1}$}

The computation of each $\omega_n$ is done by a Turing Machine (TM) which is a formalization of algorithms, see for example \cite{Turing36,Odifreddi92,MP22}. In our context, a TM is defined by a 5-uplet $(Q, q_0, q_f, \Gamma,\phi)$ where
\begin{itemize}
	\item $Q$ is the finite set of states of the head.
	\item $q_0 \in Q$ is the initial state of the head.
	\item $q_f \in Q$ is the final state of the head.
	\item $\Gamma$ is the working alphabet, containing a blank symbol $B$. The result of the computation is written using symbols from $\Gamma$, and so in our case $\A \subset \Gamma$.
	\item $\phi : Q \times \Gamma \to Q \times \Gamma \times \{\leftarrow, \rightarrow\} $ is the transition function.
\end{itemize}

A TM consists of an head in a state $q$ and placed on a cell of an infinite string indexed by $\N$. At each iteration, the head reads the symbol $\gamma$ written on its cell, and acts accordingly to $\phi(q,\gamma) = (q^\prime, \gamma^\prime, x)$: the head replace the symbol $\gamma$ by $\gamma^\prime$ on its cell, updates its state to $q^\prime$ and takes a single step in the $x$ direction. It begins initially in the state $q_0$ at position $0$ and iterates until it reaches state $q_f$.

A Turing Machine takes a word $u\in(\Gamma\backslash \{B\})^*$ as an input if it is written initially on the string in the cells $0$ to $|u|-1$. The sequence $(\omega_n)$ is said to be computed by the TM if, given the input $n$, the TM attains state $q_f$ in a finite amount of iterations with only the word $\omega_n$ written on the tape, at the leftmost positions.

The implementation in a CA is natural: the cells of the grid corresponds with the cells of the TM, and are filled with symbols from $\Gamma \cup \Gamma \times Q$, where the only cell with a symbol from $\Gamma \times Q$ represents the head of the machine. The cellular automaton acts as the identity on the cells with symbols from $\Gamma$ and follows the transition rule $\phi$ by changing the symbol and moving the head.

In our case, the TM is simulated in the computation zone, which is not infinite. But as each $\omega_n$ is computed in a finite amount of time, the TM can only moves a finite amount of steps to the right: as the size of the computation zone is $\log(t)$ after $t$ iterations and thus is ultimately enough. Moreover, we can suppose each $\omega_n$ to be computed in less space than $\log(T_{n-1})$ and time less than $T_n - T_{n-1}$.

\subsubsection{Copying the $\omega_n$}\label{subsec:copyingOmega}
In order to display one by one the symbols of $\omega_n$ we use loops inside the computing zone.

\begin{figure}[h]
    \centering
    \includegraphics[scale= 0.5]{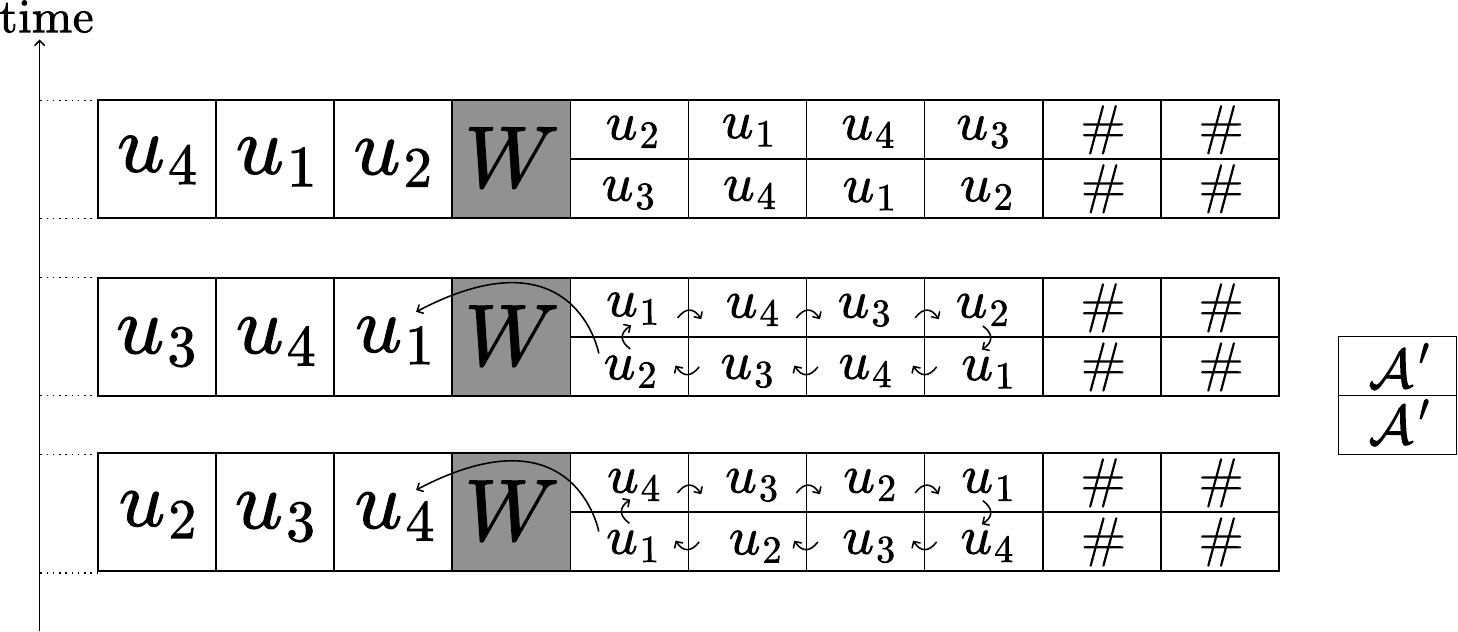}
    \caption{The displaying process. Left to the wall $W$ is the displaying zone, where the symbols are shifted to the left each step. Right to the wall is the computation zone, where the word $\omega = u_1 u_2 u_3 u_4$ is written twice. At each step, the symbols are shifted accordingly so that it forms a loop.}
\end{figure}

In order to have a smooth transition from $\omega_n$ to $\omega_{n+1}$, we consider two independent loops: when one is turning and sending its symbols to the display zone, the other one is used to compute the next word. The change of loops only occurs when the second loop is ready: the next word is computed and the loop is turning. 

To go into more details, the two loops are represented by two identical layers: $\A_1$ and $\A_2$. Each is composed by four layers: 
\[
\A_i = \A^\prime \times \A^\prime \times \A_\text{computations} \times \{0,1\} 
\]
where $\A^\prime = \A \sqcup \{\#\}$ and $\A$ is the base alphabet. The words $\omega_n$ are written using the symbols of $\A$, and the $\#$ symbol represent the end of the word. The final layer $\{0,1\}$ indicates if the loop iterate or not.

The construction of a $\omega_n$ on $\A_1$ follows the following steps: assume that initially, $t=T_{n-1}$ and the word $\omega_{n-2}$ is written on the first layer $\A^\prime$ of $\A_1$. At the end of each step, a signal is sent to the cell against the wall to continue to the next step. All signals and computations are hidden in the $\A_\text{computations}$ layer, which itself can be decomposed into different layers.
\begin{enumerate}
	\item A signal from the wall is sent away to switch the final layer to $0$, stopping the cycle. It also erase the symbols of $\omega_{n-2}$ on the way, until it reaches a $\#$. The same signal adds $2$ to the input in $\A_\text{computations}$ so that $n$ is written.
	\item A Turing machine takes $n$ as an input and computes $\omega_n$ on the first $\A^\prime$ layer, as detailed in the previous section.
	\item A Turing machine copies it on the second $\A^\prime$ layer in the reverse order.
	\item All the cells of $\A_1$ before a $\#$ symbol have their final layer switched to $1$: the loop begins to turn.
\end{enumerate}

The last step must be done simultaneously on all the cells before the $\#$ symbol so that all symbols begins to be shifted at the same time: if they are not synchronized information could be lost. Moreover, the length of the loop (the length of $\omega_n$) is not known in advance. The problem of activating simultaneously an arbitrary number of cells with a cellular automaton is known as the \define{firing squad synchronization problem} and has already been solved and optimized (see for example \cite{Mazoyer87}). We can implement such a solution within the $\A_\text{computations}$ layer.

After the construction, the cell against the wall is switch to a ``ready" state. It indicates that when the next $T_{n}$ occurs, we can switch from $\A_2$ to $\A_1$ (or vice-versa), at which point it goes back a ``waiting" state, until the next word is computed. The wall itself has another layer $\{1,2\}$ that simply indicates whether the symbol going through it should be from layer $\A_1$ or from layer $\A_2$. It switches from one value to the other when $t=T_n$, notified by the signals from the previous subsection, if the cell next to it is in a ``ready" state.

\subsection{Collisions}\label{subsec:Collisions}

When an error occurs and produces a $*$ symbol, walls are created and erase everything else on their path. The main rule the CA follows is: when two walls collide only the youngest one survives, while the oldest is gradually erased.

This idea is implemented and precisely described in \cite{HS18}: each wall has a counter with it counting its age. When two collides, the comparison is made bit by bit, until one its found to be younger than the other. As the counters are binary they take only a logarithmic amount of space, but during the comparison the wall is stopped. In \cite{ENT23}, the authors use signals and the distance between them to represent the age of each wall and make the comparisons with faster signals. Here the wall is not stopped during the comparison, but the counters take a linear amount of space (they behave like a unary counter).

In order to simplify further computations, we will mix those solutions to have a comparison method that is both quick enough that the wall does not stop and using a binary counter so that the space taken is only logarithmic.

\begin{figure}[h]
	\centering
	\begin{subfigure}[b]{0.49\textwidth}
		\centering
		\includegraphics[width=\textwidth]{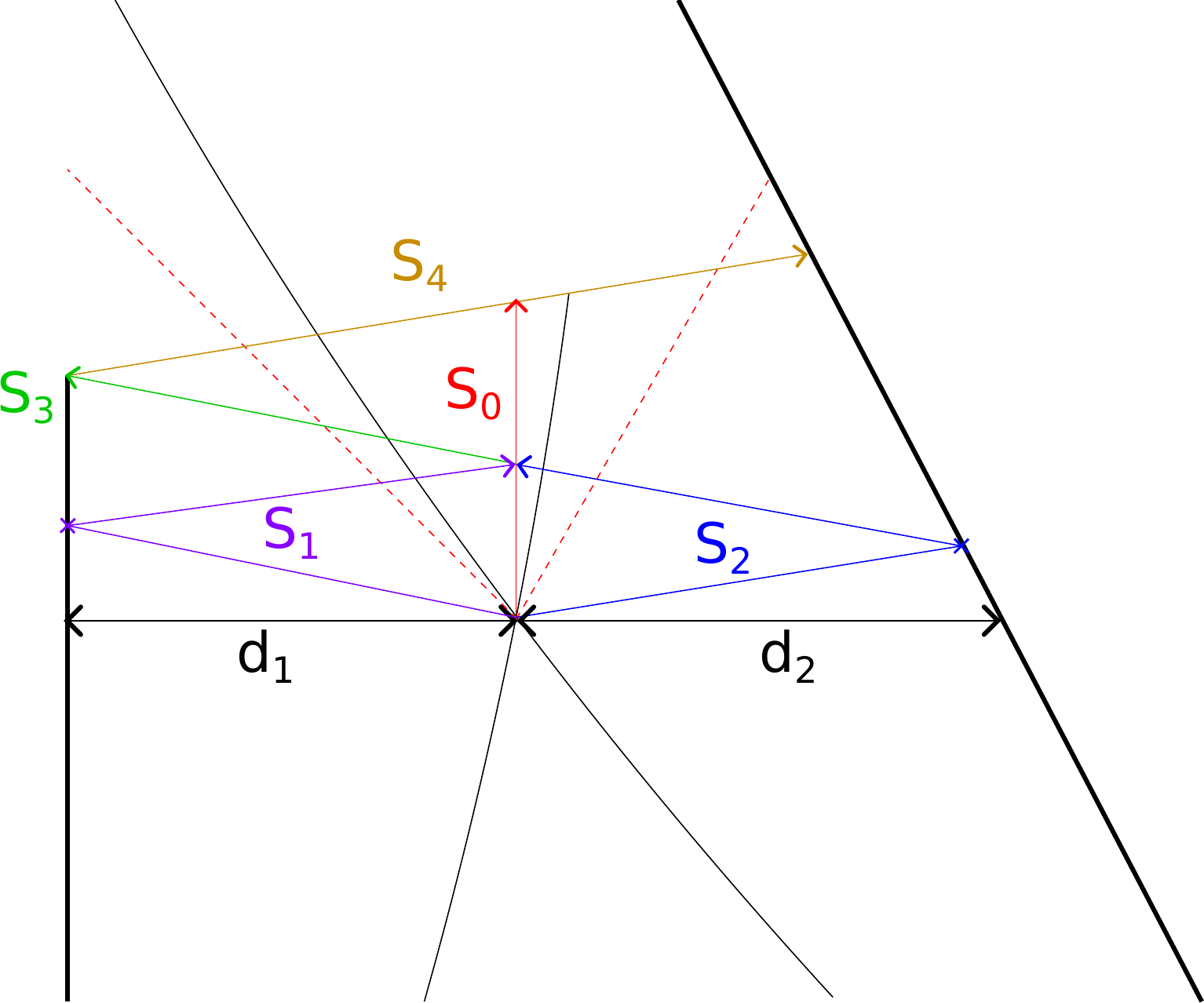}
		\caption{Comparison when the right wall is younger.}
	\end{subfigure}
	\hfill
	\begin{subfigure}[b]{0.49\textwidth}
		\centering
		\includegraphics[width=\textwidth]{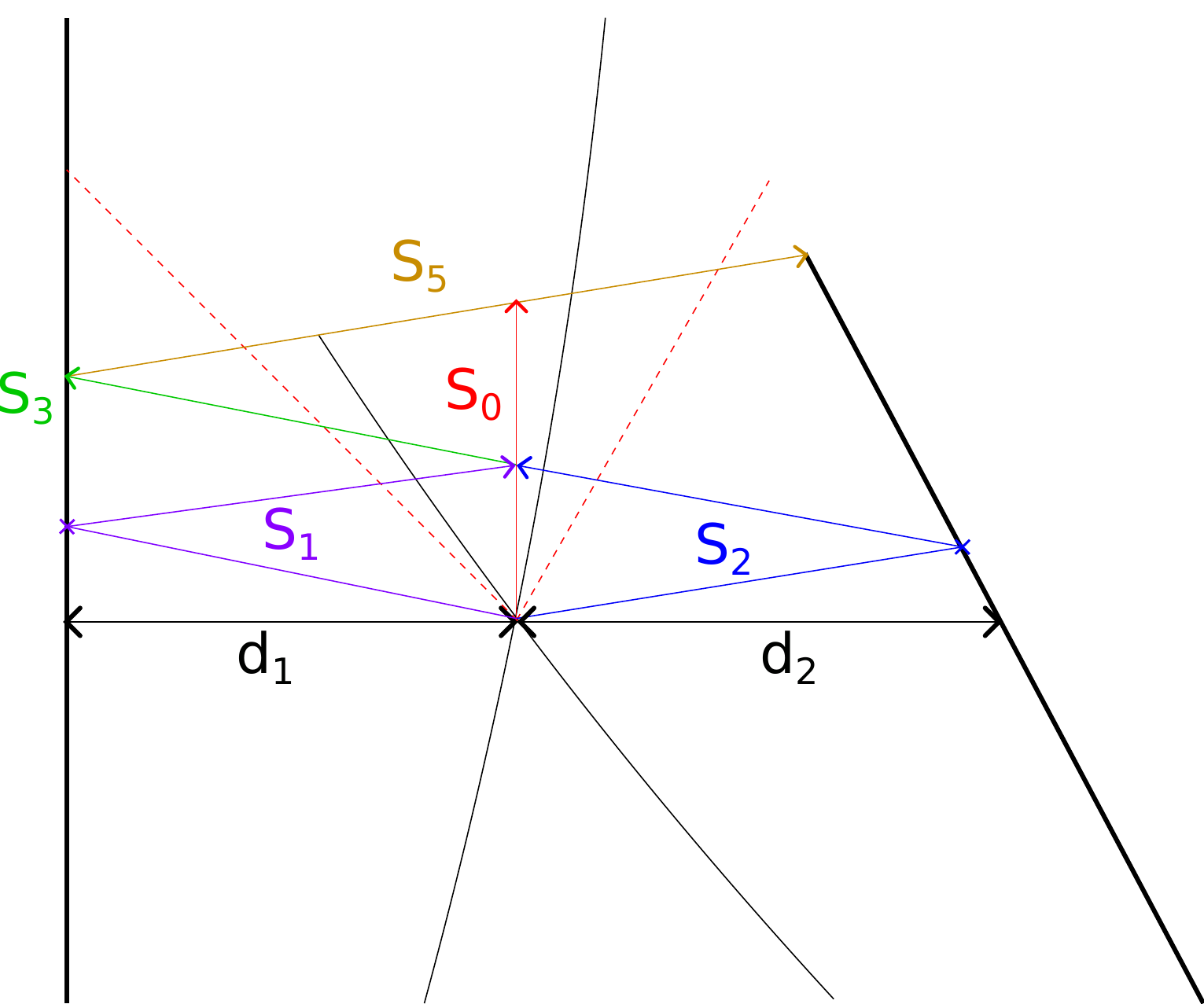}
		\caption{Comparison when the left wall is younger.}
	\end{subfigure}
	\caption{The successive signals created during a comparison when the sizes of the counters ($d_1$ and $d_2$) are equal. The dotted red lines are the theoretical maximum speed at which each counter can go. The comparison must be complete before one intersect the opposite wall.}
\end{figure}

The comparison begins when two counters collide. The comparison in itself can be decomposed in two. First, signals are created such that the length of each counter is compared: if one is strictly longer, it is therefore older and can be erased. Then, if they are the same length, a bit-by-bit comparison is made. In the following, the signals $S_i$ are going at speed $v_i \geq 1$. 

\subsubsection{Collision}
\begin{itemize}
	\item A signal $S_0$ is created (with speed $v_0 = 0$), indicating the position of the collision of the two counters.
	\item At the same time, signals $S_1$ and $S_2$ are created, going in opposite directions. They bounce off the wall they encounter.
	\item When $S_2$ is passing through the right counter, the latter is copied and sent to the left counter. The shifting speed is the same as $S_2$, $v_2$.
	\item When the copied counter arrive to the left, the bit-by-bit comparison begins (in case the length of the walls are identical), still as speed $v_2$.
	\item When the signals come back to $S_0$, two cases can occur:
	\begin{itemize}
		\item If one of them came back before the other, it creates destruction signals (respectively $S^d_{1,l}$, $S^d_{1,r}$ and $S^d_{2,l}$, $S^d_{2,r}$) that destroys the other counter, computation zone and wall, $S_0$ and the copied counter. Two signals are sent (one left and one right) as the collision location is inside both computation zones. The collision is finished.
		\item If both signals came back at the same time, the counters were the same length: the comparison begins. For this to happen when counters are the same length we need $v_2 = v_1-1$. 
	\end{itemize}
\end{itemize}

\subsubsection{Comparison}
\begin{itemize}
	\item A signal $S_3$ is created going to the left. On its way it continues the bit-by-bit comparison between the left counter and the copied one. This operation is described in details in \cite{HS18} ; the only difference here is that at each steps $v_3$ comparisons are made, instead of 1.
	\item When it reaches the left wall, it can be in 3 states:
	\begin{itemize}
		\item ``$+$" (the left wall is older) or ``$=$ (same age): a destruction signal $S_4$ is created moving to the right, destroying the left construction (like $S^d_{2,r}$).
		\item ``$-$" (the right wall is older): a destruction signal $S_5$ is created moving to the right, destroying the right construction (like $S^d_{1,r}$).
	\end{itemize}
\end{itemize}

All the comparison takes place inside counters and computation zones, so we can add as many layers as we want: for the different signals, for the copied counter, for the bit-by-bit comparison, etc. In order for the comparison and removal to be made before any proper wall encounter a counter, we can take $v_1=v_3=v_4=v_5=10$ and $v_2=9$.

One fact to observe is that if that because the comparison is processed in two steps, if we suppose that there is no errors (due to noise) in a zone of size $3\log t$ after a counter, then a construction can only be erased by a younger one: when encountering an older wall with a counter of the same length, all computations are made inside the error-less zone, and when encountering an older wall with a longer counter, the signals $S^d_{i,l/r}$ erase the latter before any rogue signal can arrive to erase the youngest construction.

\subsubsection{Format counter}

In order for the comparison to take place outside the display zone, the format counter have to be outside it (in the format zone). However, the bit-by-bit comparison requires the two counters to be in the same direction: in our case, the units is on the leftmost cell. In this situation, the counter size increases when the rightmost digit is a $2$ (a carry). On the format zone, it cannot increases on the right, as there is the display zone. Instead, we use a system of a second layer which only purpose is to act as an auxiliary to move the counter one cell to the left, increasing the format zone without touching at the display zone.

To go into more details: the second layer uses the alphabet $\{0,1,2,-\}$ where $-$ is just a blank symbol. In a normal state, all cells have the blank symbol. If the rightmost digit is a $2$, the new digit is created on the second layer on the same cell. Then, at each iteration, the non $-$ symbols on the second layer are copied on the first layer, while the symbols there are copied on the second layer on the cell at their left. If the cell at their left is a well, it is shifted and the symbol is instead copied on the first layer, completing the process.Figure \ref{fig:formatCounter} illustrate the mechanism. 

The speed of the comparison signals can easily be tweaked to take into account the non-blank second layers as ``double cells".

\begin{figure}[h]

\begin{centering}
\begin{tabular}{|c|c|c|c|c|c|c|c|}
\hline 
\multirow{2}{*}{$t=14$} & \multirow{2}{*}{$W$} & $2$ & \multicolumn{1}{c|}{$0$} & \multicolumn{1}{c|}{$1$} & \multicolumn{1}{c|}{$1$} & \multirow{2}{*}{$W$} & \multirow{2}{*}{$\rightarrow2011$}\tabularnewline
\cline{3-6} \cline{4-6} \cline{5-6} \cline{6-6} 
 &  & $-$ & $-$ & $-$ & $-$ &  & \tabularnewline
\hline 
\multirow{2}{*}{$t=13$} & \multirow{2}{*}{} & \multirow{2}{*}{$W$} & \multicolumn{1}{c|}{$1$} & \multicolumn{1}{c|}{$0$} & \multicolumn{1}{c|}{$1$} & \multirow{2}{*}{$W$} & \multirow{2}{*}{$\rightarrow1201$}\tabularnewline
\cline{4-6} \cline{5-6} \cline{6-6} 
 &  &  & $2$ & $-$ & $-$ &  & \tabularnewline
\hline 
\multirow{2}{*}{$t=12$} & \multirow{2}{*}{} & \multirow{2}{*}{$W$} & \multicolumn{1}{c|}{$2$} & \multicolumn{1}{c|}{$1$} & \multicolumn{1}{c|}{$1$} & \multirow{2}{*}{$W$} & \multirow{2}{*}{$\rightarrow2101$}\tabularnewline
\cline{4-6} \cline{5-6} \cline{6-6} 
 &  &  & $-$ & $0$ & $-$ &  & \tabularnewline
\hline 
\multirow{2}{*}{$t=11$} & \multirow{2}{*}{} & \multirow{2}{*}{$W$} & \multicolumn{1}{c|}{$1$} & \multicolumn{1}{c|}{$1$} & \multicolumn{1}{c|}{$0$} & \multirow{2}{*}{$W$} & \multirow{2}{*}{$\rightarrow1101$}\tabularnewline
\cline{4-6} \cline{5-6} \cline{6-6} 
 &  &  & $-$ & $-$ & $1$ &  & \tabularnewline
\hline 
\multirow{2}{*}{$t=10$} & \multirow{2}{*}{} & \multirow{2}{*}{$W$} & \multicolumn{1}{c|}{$2$} & \multicolumn{1}{c|}{$0$} & \multicolumn{1}{c|}{$2$} & \multirow{2}{*}{$W$} & \multirow{2}{*}{$\rightarrow202$}\tabularnewline
\cline{4-6} \cline{5-6} \cline{6-6} 
 &  &  & $-$ & $-$ & $-$ &  & \tabularnewline
\hline 
\end{tabular}\caption{\label{fig:formatCounter} Some steps of the format counter. On the left the time (in decimal), and on the right the corresponding representation in binary with carry. One can read it on the counter top to bottom and left to right, ignoring the $-$ symbols. The shift to the left at each step is not represented.}
\par\end{centering}
\end{figure}

\subsection{Errors}\label{subsec:Errors}
If an error occur in the computation zone, it will certainly be fatal for the computation. However, the logarithmic size of the computation zone is small enough in comparison to the linear size of the display zone that those errors can be neglected. The same apply for the format zone.

For the errors in the display zone, we can isolate them using a neutral symbol $\#$. Define $\A^\prime\coloneqq\A\cup\{\#\}$ as the \define{inactive} symbols and $\B\backslash\A^\prime$ the active symbols, such that if an active symbol is \define{isolated}, that is to say if it is the only active symbol in his direct neighborhood, it is shifted and becomes $\#$ at the next iteration. That way, you need at least two simultaneous errors for the display zone to break.

\subsection{Approximating the invariant measure(s)}\label{subsec:calculs}
In this section, we first show that the probability of seeing sub-words of $\omega_{n_\epsilon}$ or $\omega_{n_\epsilon+1}$ tends to $1$ when $\epsilon$ is small. We then show that this implies that the distance between $\M_\epsilon$ and $\widehat{\delta_{\omega_{n_\epsilon}}}$ tends to $0$.

For any $\epsilon>0$ we denote by $\pi_\epsilon$ a $F_\epsilon$-invariant measure, and define $\left(X^t\right)_{t\in\Z}=\left(\left(X^t_i\right)_{i\in\Z}\right)_{t\in\Z}$ a stationary orbit associated: for any $t\in\Z$ and observable $E$, $P(X^{t+1}\in E) = F_\epsilon (X^t,E)$ and $X^t\sim \pi_\epsilon$.

For any pattern $\omega = X^0_{-|\omega|+1}\cdots X^0_0$, for it to be a sub-word of $\omega_{n_\epsilon}$ or $\omega_{n_\epsilon+1}$, it suffices that:
\begin{itemize}
	\item[\textbf{Event $A_\epsilon$:}] The last $*$ symbol in the positive coordinates was between $a_\epsilon$ and $b_\epsilon$ prior steps, so that it is in a display zone; we denote it by $*_0$ in the following.
	\item[\textbf{Event $B_\epsilon$:}] There was no errors in the negative coordinates nor above the line $y=-x$ in the last $b_\epsilon$ steps, so that it is not in the computation zone or a format counter of a younger construction.
	\item[\textbf{Event $C_\epsilon$:}]There is no errors in the computation and format zones of $*_0$  with some margin ($3\log t$ each side), so that each term of the sequence $(\omega_n)$ was correctly computed and there was no comparison error on the way.
	\item[\textbf{Event $D_\epsilon$:}] There was no error in the diagonal $(X^{-i}_j)_{0\leq i\leq b_\epsilon, i-|\omega|\leq j\leq i }$, so that there is no errors in the display zone that was shifted. 
	\item[\textbf{Event $E_\epsilon$:}] There was no double-errors in the display zone, so that there is no possible "degenerate construction" that does not come from a $*$ symbol.
	\item[\textbf{Event $F_\epsilon$:}] If $(i,t)\in\Z^2$ denotes the position in space-time of $*_0$, then it should not be part of these zones:
	\begin{itemize}
		\item $\{ i \geq -t - T_{n_\epsilon} \}$, so that the pattern at $0$ is not composed by $\omega_n$ with $n < n_\epsilon$;
		\item $\{ t > - 2a_\epsilon\}$ or $\{ -T_{n_\epsilon +1} < t < - (T_{n_\epsilon +1} + a_\epsilon) \text{ and } i < -t - T_{n_\epsilon+1} \}$, so that we can suppose the pattern at $0$ to be chosen almost uniformly on the sequence ${}^\infty \omega_{n_\epsilon}^\infty$ or ${}^\infty \omega_{n_\epsilon +1 }^\infty$, and not just on the first symbols.
	\end{itemize}
\end{itemize}

\begin{figure}[h]
    \centering
    \includegraphics[scale=0.7]{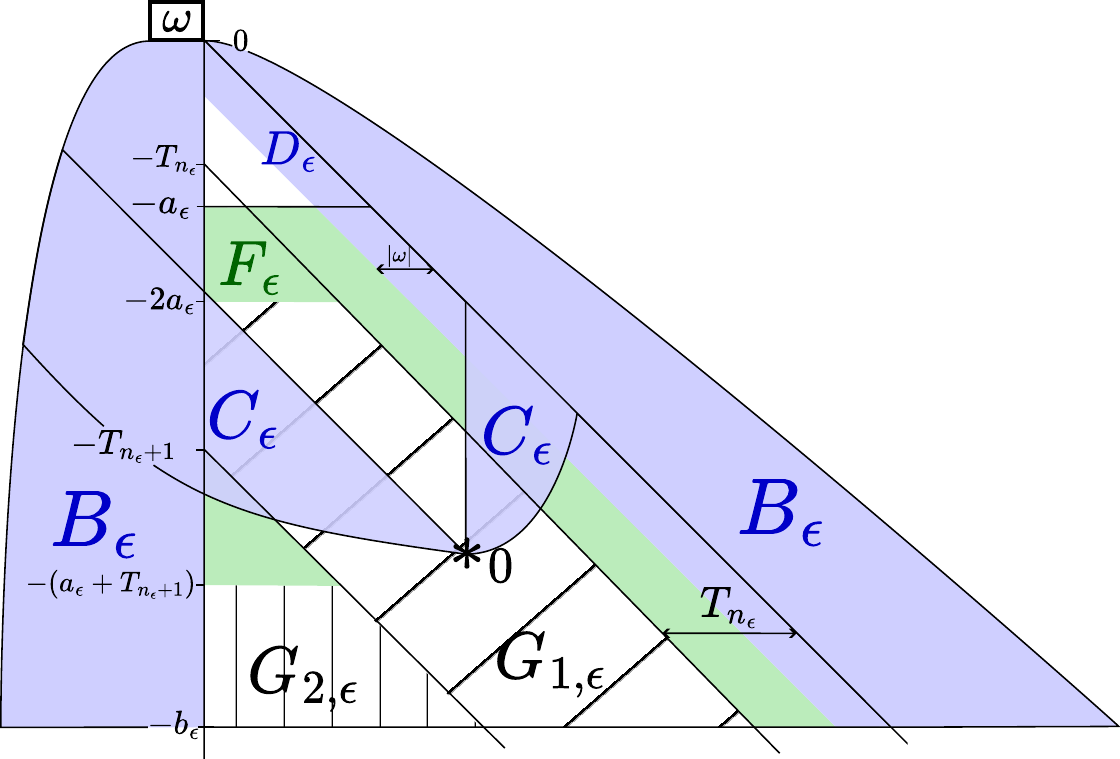}
    \caption{Graphical interpretation of the different events. There is no error in the blue areas ($B_\epsilon, C_\epsilon$ and $D_\epsilon$). $E_\epsilon$ is the absence of double errors on the whole shape. The last $*$ symbol does not appear in the green areas (event $F_\epsilon$), only in the white areas between $a_{\epsilon}$ and $b_{\epsilon}$: if it is in $G_{1,\epsilon}$ (slanted stripes), we have a pattern of $\omega_{n_\epsilon}$ at $0$, if it is in $G_{2,\epsilon}$ (vertical stripes) it will be a pattern of $\omega_{n_\epsilon+1}$.}
\end{figure}

When those $6$ events are verified, denote by $(i,t)\in\Z^2$ the position in space-time of $*_0$. We need to show that the intersection of those events tends to 1 when $\epsilon$ is small. Conditioned on their intersection, we know that $t\in [-b_\epsilon,-a_\epsilon]\subset [-T_{n_\epsilon+2}, -T_{n_\epsilon}]$ and $i\in[0,-t-T_{n_\epsilon-1}]$. Define the two cases $G_{1,\epsilon} = \{ i > -t -T_{n_\epsilon+1} \}$ and $G_{2,\epsilon} = \{i \leq -t -T_{n_\epsilon+1} \}$ such that $A_\epsilon\cap B_\epsilon\cap C_\epsilon\cap D_\epsilon\cap E_\epsilon\cap F_\epsilon = G_{1,\epsilon}\sqcup G_{2,\epsilon}$, and 
\[
p_{\epsilon}^{i}\coloneqq P\left(G_{i,\epsilon}\mid G_{1,\epsilon}\sqcup G_{2,\epsilon}\right).
\]

\begin{rem}
As there can be $\epsilon>0$ such that $G_{2,\epsilon}=\emptyset$, we have to be careful with the conditional probabilities.
\end{rem}
\begin{prop}\label{prop.convergencePiepsilon}
Let $\mu_{\epsilon}\coloneqq p_{\epsilon}^{1}\widehat{\delta_{\omega_{n_{\epsilon}}}}+p_{\epsilon}^{2}\widehat{\delta_{\omega_{n_{\epsilon}+1}}}$
(convex combination). Then 
\[
d_\M\left(\pi_{\epsilon},\mu_{\epsilon}\right)\tendsto{\epsilon}00.
\]
\end{prop}

\subsubsection{The zone where $*_0$ appears}

In this section, we show that
\[
P\left(G_{1,\epsilon}\sqcup G_{2,\epsilon}\right)\tendsto{\epsilon}01.
\]

By the choice of $a_{\epsilon}=\left(\frac{1}{\epsilon}\right)^{\frac{1}{2}-\eta}$ and $b_{\epsilon}=\left(\frac{1}{\epsilon}\right)^{\frac{1}{2}+\eta}$, we already have $P(A_\epsilon)\tendsto{\epsilon}01$. For the other ones:
\begin{itemize}
\item $P(B_\epsilon)\geq(1-\epsilon)^{b_{\epsilon}(2\log(b_{\epsilon})+|\omega|)}=\exp{(b_{\epsilon}(2\log(b_{\epsilon})+|\omega|)\log(1-\epsilon))}\tendsto{\epsilon}01$. 
\item $P(C_\epsilon\mid A_\epsilon)\geq(1-\epsilon)^{6b_{\epsilon}\log(b_{\epsilon})}\tendsto{\epsilon}01$.
\item $P(D_\epsilon)\geq(1-\epsilon)^{b_{\epsilon}|\omega|}\tendsto{\epsilon}01$.
\item The two errors must be simultaneous and close to each other, so the computation boils down to $P(E_\epsilon)\geq(1-\epsilon^{2})^{b_{\epsilon}^{2}}\tendsto{\epsilon}01$.
\item For $F_\epsilon$, we use the following equivalences (recall that the probability of having a $*$ symbol on a given cell is $\boldepsilon \coloneqq \frac{\epsilon}{|\B|}$):
\[
1-\left(1-\boldepsilon\right)^{2a_{\epsilon}} \underset{\epsilon\to0}{\sim} -2a_{\epsilon}\ln\left(1-\boldepsilon\right) \underset{\epsilon\to0}{\sim} 2a_{\epsilon}\boldepsilon = \frac{2}{|\mathcal{B}|}\epsilon^{\frac{1}{2}+\eta}
\]
and
\[
\sum_{s=0}^{\infty}(1-\boldepsilon)^{s^{2}/2} = \sum_{s=0}^{\infty}\sqrt{1-\boldepsilon}^{s^{2}} \underset{\epsilon\to0}{\sim} \frac{C}{\sqrt{1-\sqrt{1-\boldepsilon}}} \underset{\epsilon\to0}{\sim} \frac{C^{\prime}}{\sqrt{\boldepsilon}}
\]
to bound 
\begin{align*}
P\left(\overline{F_\epsilon}\right) & = \sum_{s=a_{\epsilon}}^{2a_{\epsilon}} (1-\boldepsilon)^{s^{2}/2}\sum_{t=1}^{s}(1-\boldepsilon)^{t}\boldepsilon + \sum_{s=2a_{\epsilon}}^{b_{\epsilon}} (1-\boldepsilon)^{s^{2}/2+s-T_{n_{\epsilon}}}\sum_{t=1}^{T_{n_{\epsilon}}}(1-\boldepsilon)^{t}\boldepsilon \\
    & \qquad\qquad +\sum_{s=T_{n_{\epsilon}+1}}^{a_{\epsilon}+T_{n_{\epsilon}+1}} (1-\boldepsilon)^{s^{2}/2}\sum_{t=1}^{s-T_{n_{\epsilon}+1}}(1-\boldepsilon)^{t}\boldepsilon \\
	& \leq 3\left(\sum_{s=0}^{+\infty}(1-\boldepsilon)^{s^{2}/2}\right)\left(1-(1-\boldepsilon)^{2a_{\epsilon}}\right) \\
	& \underset{\epsilon\to0}{\sim} \frac{6C^{\prime}}{\sqrt{|\mathcal{B}|}}\epsilon^{\eta}\tendsto{\epsilon}00.
\end{align*}
\end{itemize}

\subsubsection{$G_{1,\epsilon}$}
For the two following sections, suppose $u\in\mathcal{B^{*}}$ is a fixed pattern. In this section, we show that
\[
\left|P(\omega=u \mid G_{1,\epsilon})-\widehat{\delta_{\omega_{n_{\epsilon}}}}([u])\right|\tendsto{\epsilon}00.
\]

We can decompose $G_{1,\epsilon}$ as $G_{1,\epsilon}=\widetilde{G_{1,\epsilon}}\sqcup I$ where $I$ is the event where the symbols displayed at $t=0$ on $\omega$ are a mix between the ones of $\omega_{n_{\epsilon-1}}$ and $\omega_{n_{\epsilon}}$; this happens when $*_0$ (with coordinate $(i,t)\in\Z^2$) is in an interface of size $\left|u\right|$ along the line $x+y=-T_{n_{\epsilon}}$. Given $G_{1,\epsilon}$ and $|t| = s$, this corresponds to the event $\{i > s - T_{n_\epsilon} - |u|\}$. Using
\[
l(s) \coloneqq \min\left(s,T_{n_{\epsilon}+1}\right)-\left(T_{n_{\epsilon}}+\left|u\right|\right)
\]
the length of $\widetilde{G_{1,\epsilon}}$ at $|t| = s$ and $\boldepsilon = \frac{\epsilon}{|\B|}$ the probability of having a $*$ symbol in a given cell, we can compute $P\left(I\mid G_{1,\epsilon}\cap|t|=s\right) =  \frac{(1-\boldepsilon)^{l(s)}\left(1-(1-\boldepsilon)^{\left|u\right|}\right)}{1-(1-\boldepsilon)^{l(s) + \left|u\right|}}$: the probability that when a $*$ appears in a zone of size $l(s)+|u|$, it only appears in the last $|u|$ cells. Therefore,
\begin{align*}
P\left(I\mid G_{1,\epsilon}\right) & =\sum_{s=2a_{\epsilon}}^{b_{\epsilon}} P\left(I\mid G_{1,\epsilon}\cap|t|=s\right)P\left(|t|=s\mid G_{1,\epsilon}\right)\\
 & \leq\sum_{s}\frac{(1-\boldepsilon)^{l(s)}\left(1-(1-\boldepsilon)^{\left|u\right|}\right)}{1-(1-\boldepsilon)^{l(s) + \left|u\right|}}P\left(|t|=s\mid G_{1,\epsilon}\right)\\
 & \leq\frac{1-(1-\boldepsilon)^{\left|u\right|}}{1-(1-\boldepsilon)^{a_{\epsilon}}}\tendsto{\epsilon}00.
\end{align*}

Therefore, we only have to demonstrate that 
\[
\left|P\left(\omega=u\mid\widetilde{G_{1,\epsilon}}\right)-\widehat{\delta_{\omega_{n_{\epsilon}}}}([u])\right|\tendsto{\epsilon}00.
\]

We can replace $\widehat{\delta_{\omega_{n_{\epsilon}}}}([u])$ by the following approximation.
\begin{prop}
Define $J_{n}^{N}=\left\{ 1\leq j\leq N\mid\left(\omega_{n}\right)_{\left\llbracket N-j+1;N-j+\left|\omega\right|\right\rrbracket }^{\infty}=u\right\} $ and $l(s)=\min\left(s,T_{n_{\epsilon}+1}\right)-\left(T_{n_{\epsilon}}+\left|u\right|\right)$.
Then
\[
\left|\sum_{s=2a_{\epsilon}}^{b_{\epsilon}}\frac{\left|J_{n_{\epsilon}}^{l(s)}\right|}{l(s)}P\left(|t|=s\mid\widetilde{G_{1,\epsilon}}\right)-\widehat{\delta_{\omega_{n_{\epsilon}}}}([u])\right|\tendsto{\epsilon}00.
\]
\end{prop}

\begin{rem}
For $j\in J_n^N$, $N-j+1$ is a rank where the pattern $u$ appear in $\omega_n$, in the first $N$ terms. Given $\widetilde{G_{1,\epsilon}}$ and $|t|=s$, $*_0$ must be at a position $j\in J_n^{l(s)}$ away from the left border of $G_{1,\epsilon}$ to have the pattern $u$ in $\omega$ at time $0$.
\end{rem}

\begin{proof}
As we compute $\omega_{n_{\epsilon}}$ in $T_{n_{\epsilon}}$ steps, we have $\left|\omega_{n_{\epsilon}}\right|\leq\ln\left(T_{n_{\epsilon}}\right)\leq\ln\left(a_{\epsilon}\right)$, therefore $\left|\omega_{n_{\epsilon}}\right|=o\left(a_{\epsilon}\right)$. As $\frac{\left|J_{n_{\epsilon}}^{l}\right|}{l}$ corresponds to the empirical measure of the frequency of pattern $u$ in $\omega_{n_{\epsilon}}^{\infty}$ in the first $l$ symbols, we have the approximation 
\[
\left|\frac{\left|J_{n_{\epsilon}}^{l}\right|}{l}-\widehat{\delta_{\omega_{n_{\epsilon}}}}([u])\right|\leq\frac{\left|\omega_{n_{\epsilon}}\right|}{l}.
\]
Thus 
\begin{align*}
\left|\underset{\text{convex combination}}{\underbrace{\sum_{s=2a_{\epsilon}}^{b_{\epsilon}}\frac{\left|J_{n_{\epsilon}}^{l(s)}\right|}{l(s)} P\left(|t|=s\mid\widetilde{G_{1,\epsilon}}\right)}} - \widehat{\delta_{\omega_{n_{\epsilon}}}}([u])\right| & \leq\sum_{s=2a_{\epsilon}}^{b_{\epsilon}}\left|\frac{\left|J_{n_{\epsilon}}^{l(s)}\right|}{l(s)} -\widehat{\delta_{\omega_{n_{\epsilon}}}}([u])\right|P\left(|t|=s\mid\widetilde{G_{1,\epsilon}}\right) \\
	& \leq\frac{\left|\omega_{n_{\epsilon}}\right|}{l\left(2a_{\epsilon}\right)}\leq\frac{\left|\omega_{n_{\epsilon}}\right|}{a_{\epsilon}}\tendsto{\epsilon}00.
\end{align*}
\end{proof}
Let us decompose the event along the different lines (the time $t$
at which $*_{0}$ appeared).

\begin{align*}
P\left(\omega=u\mid\widetilde{G_{1,\epsilon}}\right) & = \sum_{s=2a_{\epsilon}}^{b_{\epsilon}}P\left(\omega=u\mid\widetilde{G_{1,\epsilon}}\cap|t|=s\right)P\left(|t|=s\mid\widetilde{G_{1,\epsilon}}\right)\\
 & = \sum_{s=2a_{\epsilon}}^{b_{\epsilon}}\sum_{k\in J_{n_{\epsilon}}^{l(s)}}\frac{(1-\boldepsilon)^{k-1}\boldepsilon}{1-(1-\boldepsilon)^{l(s)}}P\left(|t|=s\mid\widetilde{G_{1,\epsilon}}\right).
\end{align*}
Using the last proposition, we can look at the difference with the approximation of $\widehat{\delta_{\omega_{n_{\epsilon}}}}([u])$:
\begin{align*}
\left|P(\omega=u\mid \widetilde{G_{1,\epsilon}}) - \sum_{s=2a_{\epsilon}}^{b_{\epsilon}}\frac{|J_{n_{\epsilon}}^{l(s)}|}{l(s)}P(|t|=s\mid\widetilde{G_{1,\epsilon}})\right| & \leq \sum_{s=2a_{\epsilon}}^{b_{\epsilon}} \sum_{k\in J_{n_{\epsilon}}^{l(s)}} \left|\frac{(1-\boldepsilon)^{k-1}\boldepsilon}{1-(1-\boldepsilon)^{l(s)}}-\frac{1}{l(s)}\right| P(|t|=s\mid\widetilde{G_{1,\epsilon}}) \\
	& \leq\max_{2a_{\epsilon}\leq s\leq b_{\epsilon}} \max_{k} \left|\frac{(1-\boldepsilon)^{k-1}\boldepsilon l(s)}{1-(1-\boldepsilon)^{l(s)}}-1\right| \cdot \underset{\leq1}{\underbrace{\frac{\left|J_{n_{\epsilon}}^{l(s)}\right|}{l(s)}}} \\
	& \leq\max_{\underset{k\leq l(s)}{2a_{\epsilon}\leq s\leq b_{\epsilon}}}\left|\frac{(1-\boldepsilon)^{k-1}\boldepsilon l(s)}{1-(1-\boldepsilon)^{l(s)}}-1\right|.
\end{align*}
Finally, the power series of the exponential function gives
\begin{align*}
\left|(1-\boldepsilon)^{k-1}\boldepsilon l(s)-\left(1-(1-\boldepsilon)^{l(s)}\right)\right| & \leq\left|(1-\boldepsilon)^{k-1}\boldepsilon l(s) + l(s)\ln(1-\boldepsilon)\right| \\ & \qquad\qquad + \left|\frac{l(s)^{2}\ln(1-\boldepsilon)^{2}}{2}\sum_{k\geq2}\frac{l(s)^{k-2}\ln(1-\boldepsilon)^{k-2}}{k!}\right|\\
 & =\boldepsilon l(s)\left|(1-\boldepsilon)^{k-1} + \frac{\ln(1-\boldepsilon)}{\boldepsilon}\right| \\ & \qquad\qquad + \left|\frac{l(s)^{2}\ln(1-\boldepsilon)^{2}}{2}\sum_{k\geq2}\frac{l(s)^{k-2}\ln(1-\boldepsilon)^{k-2}}{k!}\right|
\end{align*}
and we can conclude by the monotony of $x\mapsto\frac{\boldepsilon x}{1-(1-\boldepsilon)^{x}}$ and $x\mapsto\frac{-x}{1-e^{x}}$ to conlude that for all $2a_{\epsilon}\leq s\leq b_{\epsilon}$ and $k\leq l(s)$:
\begin{align*}
\left|\frac{(1-\boldepsilon)^{k-1}\boldepsilon l(s)}{1-(1-\boldepsilon)^{l(s)}}-1\right| & \leq\frac{\boldepsilon l(s)}{1-(1-\boldepsilon)^{l(s)}}\left|(1-\boldepsilon)^{l(s)} + \frac{\ln(1-\boldepsilon)}{\boldepsilon}\right| \\ & \qquad\qquad + \left|\frac{l(s)\ln(1-\boldepsilon)}{1-(1-\boldepsilon)^{l(s)}}\right|\left|\frac{l(s)\ln(1-\boldepsilon)}{2}\right|\left|\sum_{k\geq2}\frac{l(s)^{k-2}\ln(1-\boldepsilon)^{k-2}}{k!}\right|\\
 & \leq\underset{\tendsto{\epsilon}01}{\underbrace{\frac{\boldepsilon b_{\epsilon}}{1-(1-\boldepsilon)^{b_{\epsilon}}}}} \underset{\tendsto{\epsilon}00}{\underbrace{\left|(1-\boldepsilon)^{b_{\epsilon}} + \frac{\ln(1-\boldepsilon)}{\boldepsilon}\right|}} \\ & \qquad\qquad + \underset{\tendsto{\epsilon}01}{\underbrace{\left|\frac{b_{\epsilon}\ln(1-\boldepsilon)}{1-(1-\boldepsilon)^{b_{\epsilon}}}\right|}}\underset{\tendsto{\epsilon}00}{\underbrace{\left|\frac{b_{\epsilon}\ln(1-\boldepsilon)}{2}\right|}}\underset{\leq1}{\underbrace{\left|\sum_{k\geq2}\frac{a_{\epsilon}^{k-2}\ln(1-\boldepsilon)^{k-2}}{k!}\right|}}\\
 & \tendsto{\epsilon}00.
\end{align*}

The bound is uniform in $s$, so we can conclude that 
\[
\left|P(\omega=u\mid \widetilde{G_{1,\epsilon}}) - \sum_{s=2a_{\epsilon}}^{b_{\epsilon}}\frac{|J_{n_{\epsilon}}^{l(s)}|}{l(s)}P(|t|=s\mid\widetilde{G_{1,\epsilon}})\right| \tendsto{\epsilon}{0}0.
\]

\subsubsection{$G_{2,\epsilon}$}
In this section we show that
\[
\left|P\left(\omega=u\cap G_{2,\epsilon}\mid G_{1,\epsilon}\sqcup G_{2,\epsilon}\right)-p_{\epsilon}^{2}\widehat{\delta_{\omega_{n_{\epsilon}}+1}}([u])\right|\tendsto{\epsilon}00.
\]
Recall that $p_{\epsilon}^{2} = P\left(G_{2,\epsilon}\mid G_{1,\epsilon}\sqcup G_{2,\epsilon}\right)$.

We consider the two following cases depending on $\epsilon$: $T_{n_{\epsilon}+1}\geq\frac{b_{\epsilon}}{2}$ and $T_{n_{\epsilon}+1}<\frac{b_{\epsilon}}{2}$. In the first case, it suffices to notice that $G_{2,\epsilon}$ implies that $*_0$ cannot appear at time between $-T_{n_\epsilon+1}$ and $-2a_\epsilon$, so that 
\[
p_{\epsilon}^{2} \leq (1-\boldepsilon)^{\frac{\left(T_{n_{\epsilon}+1} -2a_\epsilon\right)^{2}}{2}} \leq (1-\boldepsilon)^{\frac{\left(\frac{b_{\epsilon}}{2}-2a_\epsilon\right)^{2}}{2}} \tendsto{\epsilon}00
\]
with $\boldepsilon = \frac{\epsilon}{|\B|}$. In the second case, we have $p_{\epsilon}^{2}>0$ so we can condition by $G_{2,\epsilon}$. The computations are analogous of the ones for $G_{1,\epsilon}$, and we can conclude by
\[
\left|P\left(\omega=u\cap G_{2,\epsilon}\mid G_{1,\epsilon}\sqcup G_{2,\epsilon}\right)-p_{\epsilon}^{2}\widehat{\delta_{\omega_{n_{\epsilon}}+1}}([u])\right|\leq\begin{cases}
(1-\boldepsilon)^{\frac{\left(\frac{b_{\epsilon}}{2}-2a_\epsilon\right)^{2}}{2}} & \text{if }T_{n_{\epsilon}+1}\geq\frac{b_{\epsilon}}{2},\\
\left|P\left(\omega=u\mid G_{2,\epsilon}\right)-\widehat{\delta_{\omega_{n_{\epsilon}+1}}}([u])\right| & \text{otherwise}.
\end{cases}
\]

\subsubsection{Proof of Proposition~\ref{prop.convergencePiepsilon}}

Recall that $\pi_\epsilon$ is a given $F_\epsilon$ invariant measure and $\mu_\epsilon$ is defined by $\mu_{\epsilon} = p_{\epsilon}^{1}\widehat{\delta_{\omega_{n_{\epsilon}}}} + p_{\epsilon}^{2}\widehat{\delta_{\omega_{n_{\epsilon}+1}}}$, with 
$p_\epsilon^i = P\left(G_{i,\epsilon}\mid G_{1,\epsilon}\sqcup G_{2,\epsilon}\right)$.

\begin{proof}
Let $u\in\mathcal{B^{*}}$ be a given pattern. We want to show that $\left|\pi_{\epsilon}([u])-\mu_{\epsilon}([u])\right|\tendsto{\epsilon}00$. Let us decompose 
\[
\pi_{\epsilon}([u]) = P(\omega=u) = P(\omega=u\mid G_{1,\epsilon}\sqcup G_{2,\epsilon}) P(G_{1,\epsilon}\sqcup G_{2,\epsilon}) + P(\omega=u\mid\overline{G_{1,\epsilon}\sqcup G_{2,\epsilon}}) P(\overline{G_{1,\epsilon}\sqcup G_{2,\epsilon}}),
\]
which leads to
\begin{align*}
\left|\pi_{\epsilon}([u])-\mu_{\epsilon}([u])\right| & \leq \left|P(\omega=u\cap G_{1,\epsilon}\mid G_{1,\epsilon}\sqcup G_{2,\epsilon})-p_{\epsilon}^{1}\widehat{\delta_{\omega_{n_{\epsilon}}}}([u])\right| \\
	& \qquad\qquad + \left|P(\omega=u\cap G_{2,\epsilon}\mid G_{1,\epsilon}\sqcup G_{2,\epsilon})-p_{\epsilon}^{2}\widehat{\delta_{\omega_{n_{\epsilon}}+1}}([u])\right| + P\left(\overline{G_{1,\epsilon}\sqcup G_{2,\epsilon}}\right) \\
	& =p_{\epsilon}^{1}\left|P\left(\omega=u\mid G_{1,\epsilon}\right)-\widehat{\delta_{\omega_{n_{\epsilon}}}}([u])\right| \\
	& \qquad\qquad + \left|P\left(\omega=u\cap G_{2,\epsilon}\mid G_{1,\epsilon}\sqcup G_{2,\epsilon}\right)-p_{\epsilon}^{2}\widehat{\delta_{\omega_{n_{\epsilon}}+1}}([u])\right| + P\left(\overline{G_{1,\epsilon}\sqcup G_{2,\epsilon}}\right)
\end{align*}
and each of these three terms converges towards 0 when $\epsilon\to0$ by the results of last three sections.
\end{proof}

\subsection{Conclusion} \label{subsec:conclusionPreuve}

As a result of Proposition \ref{prop.convergencePiepsilon}, we can finish the proof of Theorem \ref{thm:realization}. 

For a fixed sequence $\epsilon_{i}\tendsto i{\infty}0$ and a choice of  $\pi_{\epsilon_{i}}\in\M_{\epsilon_{i}}$, we get $d_\M\left(\pi_{\epsilon_i},\mu_{\epsilon_i}\right)\tendsto{i}{+\infty}0$, so that $(\pi_{\epsilon_i})$ and $(\mu_{\epsilon_i})$ have the same accumulation points. By definition, $\mu_{\epsilon} \in \text{Conv}\left(\widehat{\delta_{\omega_{n_{\epsilon}}}},\widehat{\delta_{\omega_{n_{\epsilon+1}}}}\right)$, and so
\[
\text{Acc}\left(\pi_{\epsilon_{i}}\right)\subset\text{Acc}\left(\text{Conv}\left(\widehat{\delta_{\omega_{n_{\epsilon_{i}}}}},\widehat{\delta_{\omega_{n_{\epsilon_{i}}+1}}}\right)\right)\subset \mathcal{K}.
\]

Reciprocally, if $\mu\in \mathcal{K}$,  then there exists $n_{j}\tendsto j{\infty}\infty$ such that  $\mu=\lim_{j\to\infty}\widehat{\delta_{\omega_{n_{j}}}}$. Choosing $\epsilon_{j}\tendsto j{\infty}0$ such that  $n_{\epsilon_{j}}=n_{j}$, then for any choice of $\pi_{\epsilon_{j}}\in\M_{\epsilon_{j}}$, one will have $\pi_{\epsilon_{j}}\weakto j{\infty}\mu$ (using $d_\M\left(\widehat{\delta_{\omega_{n}}},\widehat{\delta_{\omega_{n+1}}}\right)\tendsto n{\infty}0$).

\section{Dependence on the family of noise}\label{section.DependenceNoise}

When the family of noise considered is not computable, the set of limit measures is not necessary $\Pi_2$-computable. For a probability vector $\alpha=(\alpha_a)_{a\in\A}$ (recall that $\alpha_a\geq 0$ for all $a\in\A$ and $\sum_{a\in\A}\alpha_a=1$), consider the random perturbation associated to the cellular automata $F$ of bias $\alpha$ denoted $(F_{\alpha,\epsilon})_{\epsilon>0}$ which is defined by the local rule

$$f_{\alpha,\epsilon}(u,a)=
\begin{cases}
 \epsilon\alpha_a &\textrm{ if }f(u)\ne a\\
1-\epsilon+\epsilon\alpha_a&\textrm{ if }f(u)=a\\
\end{cases}.
$$
In other word, this means that independently, for each cell, one applies $F$ and with probability $\epsilon$ a change is made by choosing a new symbol according to the probability vector $\alpha$. The uniform noise corresponds to the probability vector with $\alpha_a = \frac{1}{|\A|}$.

With this perturbation if $F$ is the identity then the set of $(F_{\alpha,\epsilon},\sigma)$-invariant measures is $\M_{\epsilon}(\alpha)=\{\lambda_{\alpha}\}$ where $\lambda_\alpha$ is the Bernoulli measure with $\lambda_\alpha([a])=\alpha_a$. Thus, if $\alpha$ is not limit-computable (this is possible since the set of limit-computable numbers is countable) the limit set $\Ml(\alpha)$ is not necessarily $\Pi_2$-computable.

Denote by $\mathcal{V}_{\B}=\left\{(\alpha_b)_{b\in\B}:\alpha_b>0,\ \forall b\in\B\textrm{ and }\sum_{b\in\B}\alpha_b=1\right\}$ the set of strictly positive probability vectors indexed by $\B$ and $\mathfrak{K}_{\A}$ the set of connected compact subsets of $\M\left(\A^\Z\right)$. We want to characterize the set of functions which can be obtained as $\alpha\longmapsto\Ml(\alpha)$. Here, the notion of computable function is too strong: let us introduce the weaker notion of $\Pi_2$-computable function with oracle.

\subsection{Definitions and constraints}

We are going to define computation with oracle, the definitions are taken from Section 5.4 of \cite{HS18}.

\begin{defn} \phantom{}
\begin{itemize}
	\item A \define{Turing machine with oracle} in $\mathcal{V}_{\B}$ is a Turing machine that with $\alpha \in \mathcal{V}_{\B}$ fixed before the computation. The machine can query the oracle at any time by writing $k\in\N$ on a special additional tape and entering a special oracle state. At this step the contents of the oracle tape are replaced by an approximation of $\alpha$ with an error at most of $2^{-k}$. Denote by $M_\alpha(x)$ a Turing machine with oracle $\alpha \in [0,1]^k$ and input $x$.
	\item A function $f:\mathcal{V}_{\B}\times\N \to \A^*$ is \define{computable with oracle in $\mathcal{V}_{\B}$} if there exists a Turing machine $M$ with oracle in $\mathcal{V}_{\B}$ such that $f(\alpha,n) = M_\alpha(n)$.
	\item A function $\Psi : \mathcal{V}_{\B} \to \mathfrak{K}_{\A}$ is \define{$\Pi_2$-computable with oracle in $\mathcal{V}_{\B}$} if there exists a Turing Machine with oracle in $\mathcal{V}_{\B}$ $M$ such that for all $(x,r)\in(\mathfrak{P},\Q_+)$,
		\[
			\Psi(\alpha) \cap \overline{B(x,r)} \neq \emptyset \Longleftrightarrow \forall n\in\N, \exists m\in\N, M_\alpha(x,r,m,n) = 1.
		\]
\end{itemize}

\end{defn}

\begin{prop}
Let $F:\A^\Z\longrightarrow\A^\Z$ be a cellular automaton. Then the function $\alpha\longmapsto\Ml(\alpha)$ is $\Pi_3$-computable with oracle in $\mathcal{V}_{\A}$.

Moreover, if $\Ml(\alpha)$ is uniformly approached for any $\alpha$, then the function $\alpha\longmapsto\Ml(\alpha)$ is $\Pi_2$-computable with oracle in $\mathcal{V}_{\A}$.
\end{prop}

\begin{proof}
The random perturbation $(\alpha, \epsilon) \mapsto f_{\alpha,\epsilon}$ is computable with oracle in $\mathcal{V}_\A$: we can then use the same computations as in the proofs of Proposition \ref{prop:computablePerturbation} and Proposition \ref{prop:UniformApproachPi2} to show the result.
\end{proof}

\subsection{Realization}

If we apply a bias $\alpha \in \mathcal{V}_\B$ in the construction of Theorem \ref{thm:realization}, its behavior is not changed (as long as $\alpha_* \neq 0$), as the probability of error is $\epsilon$ and the probability of having a $*$ symbol are proportional (and we suppose no errors in the computations). We would have the same $\Ml$ for any $\alpha$. In order to take it into account, we modify this construction by adding another layer where the CA acts as the identity function and out all the bias in the noise on this layer.

\begin{prop}[from \cite{HS18}]
If $\Psi: [0,1] \to \mathfrak{K}_{\A}$ is a $\Pi_2$-computable function with oracle in $[0,1]$ such that every element in $\Psi([0,1])$ is connected, then there exist a function $f:[0,1]\times\N \to \A^*$ computable with oracle in $[0,1]$ such that 
\[
	\Psi(\alpha) = \Acc{n\to\infty}{\widehat{\delta_{f(\alpha,n)}}}
\]
and $d\left(\delta_{f(\alpha,n)},\delta_{f(\alpha,n+1)} \right) \tendsto{n}{\infty}0$.
\end{prop}

\begin{thm}\label{thm:realisationFunction}
Let $\A$ be a finite alphabet and $\Psi: [0,1] \to \mathfrak{K}_{\A}$ a $\Pi_2$-computable function with oracle in $[0,1]$ such that every element in $\Psi([0,1])$ is connected. 

Then there exists an alphabet $\B \supset \A$ and a cellular automaton $G$ on $(\B\times\{0,1\})^\Z$ such that 
\[
\Ml(\beta^\alpha)= \Psi(\alpha)\times\{\lambda_\alpha\},
\]
where $\beta^\alpha_{(b,i)} = \begin{cases}
\frac{\alpha}{|\B|} & \text{if }i=1 \\
\frac{1-\alpha}{|\B|} & \text{if }i=0
\end{cases}$ for all $(b,i)\in\B\times\{0,1\}$, and $\lambda_\alpha$ is the Bernoulli measure of parameter $\alpha$ on $\{0,1\}^\Z$. Moreover $\Ml(\beta^\alpha)$ is uniformly approached.

\end{thm}
\begin{proof}
We consider a modified version of the cellular automaton of the realization theorem coupled with the identity on $\{0,1\}^\Z$. The CA is modified as follows: in the computation zone, a Turing machine ``reads" the proportion of $1$ on the second layer. The machine computing the approximating sequence takes it as a parameter to compute $f(\alpha,n)$ such that $\Psi(\alpha) = \Acc{n\to\infty}{\widehat{\delta_{f(\alpha,n)}}}$.

As in the proof for the main theorem, we can suppose $f(\alpha,n)$ to be computable in time $T_n - T_{n-1}$ and in space $\log(T_{n-1})$, but also that all the oracle calls to $\alpha$ only asks for an approximation up to an error of $2^{-n/4}$. If not, we can just repeat $f(\alpha,n-1)$ until the machine has enough time and space: this does not change the set of accumulation points $\Psi(\alpha)$.

On any invariant measure $\pi_\epsilon\in \M_\epsilon$, the projection on the second coordinate is an invariant measure of the identity perturbed by a noise of size $\epsilon$ and bias $\alpha\in[0,1]$: there is only one of them, that is the uniform Bernoulli measure of parameter $(1-\alpha,\alpha)$. Thus, the reading of the proportion of $1$ gives an approximation of $\alpha$. 
At time $T_n = 2^{2^n}$, the size of the computation zone is $m=\log(T_n) = 2^n$. Using Bienaymé-Chebychev inequality with $(X_i)$ a sequence of i.i.d random variable with distribution $\mathcal{B}(\alpha)$, one obtains $$P\left( \left|\alpha-\frac{1}{m} \sum_{i=0}^{m-1} X_i 	\right| > \frac{1}{\sqrt[4]{m}} \right) \leq \frac{\alpha(1-\alpha)}{\sqrt{m}}.$$

In the computations of the realization theorem, we can modify event $C_\epsilon$ defined in Section~\ref{subsec:calculs} to also include the event that the proportion of $1$ in the second layer of the computation zone between $T_{n_\epsilon-1}$ and $T_{n_\epsilon}$ is a $2^{-\frac{n_\epsilon}{4}}$ approximation of $\alpha$. Because we condition on the absence of errors in the computation zone, this only depends on the values on the second layer of the cells at the boundary of the computation zone. They are i.i.d. with distribution $\mathcal{B}(\alpha)$, and thus the previous inequality proves that the probability of this event tends to $1$ as $\epsilon$ tends to $0$. The rest of the proof is analogous.
\end{proof}

With this result, it is possible to construct a cellular automaton which is strongly unstable with respect to any family of perturbation with bias, more precisely we have the following corollary. 

\begin{cor}\label{Cor.Strongly Unstable}
Given a finite alphabet $\mathcal{A}$, there exists a cellular automaton on $\B\times\{0,1\}$ where $\B\supset\A$ such that for any bias $\alpha$ on the second coordinate, for any connected compact set $\mathcal{K}\subset\M(\A^\Z)$, for any $\delta>0$, there exists a bias $\alpha'$ such that $|\alpha-\alpha'|\leq\delta$ (so $||f_{\alpha,\epsilon}-f_{\alpha',\epsilon}||_{\infty}\leq\delta$ for all $\epsilon>0$) and the uniformly approached limit set  is $\mathcal{K}\times\{\lambda_{\alpha'}\}$.
\end{cor}
\begin{proof}
 Consider that $\alpha$ is written in bases $|\A|+1$ on the alphabet $\A\cup\{\$\}$ where $\$$ is the bit with higher weight. Given $\alpha$, we associate the sequence of word $(w^\alpha_i)_{i\in\N}$ such that separated by $\$$ one obtains the decimals of $\alpha$. Of course the sequence is finite if there is only finite many $\$$ in the decimals of $\alpha$. In this case the sequence is completed by the last term encountered. Consider the function
 $$\Psi:\alpha\longmapsto \bigcap_{N\in\N}\overline{\bigcup_{n\geq N}\text{Conv}\left(\widehat{\delta_{w_n^\alpha}},\widehat{\delta_{w_{n+1}^\alpha}}\right)}$$
 where $\text{Conv}\left(\widehat{\delta_{u}},\widehat{\delta_{v}}\right)=\left\{(1-t)\widehat{\delta_{u}}+t\widehat{\delta_{v}}:t\in[0,1]\right\}$. We remark that as only the asymptotic bits of the input of $\Psi$ allow to construct the output, we deduce that for all $\alpha$ and $\mathcal{K}\subset\M(\A^\Z)$, it is possible to find $\alpha'$ as close as desired such that $\Psi(\alpha')=\mathcal{K}$.
 Moreover, the function $\Psi$ is a $\Pi_2$-computable function with oracle in $[0,1]$ such that every element in $\Psi([0,1])$ is connected. The Corollary follows from Theorem~\ref{thm:realisationFunction}.
\end{proof}

\end{document}

%% file: Schemaconnexite.pdf_tex
\begingroup%
  \makeatletter%
  \providecommand\color[2][]{%
    \errmessage{(Inkscape) Color is used for the text in Inkscape, but the package 'color.sty' is not loaded}%
    \renewcommand\color[2][]{}%
  }%
  \providecommand\transparent[1]{%
    \errmessage{(Inkscape) Transparency is used (non-zero) for the text in Inkscape, but the package 'transparent.sty' is not loaded}%
    \renewcommand\transparent[1]{}%
  }%
  \providecommand\rotatebox[2]{#2}%
  \newcommand*\fsize{\dimexpr\f@size pt\relax}%
  \newcommand*\lineheight[1]{\fontsize{\fsize}{#1\fsize}\selectfont}%
  \ifx\svgwidth\undefined%
    \setlength{\unitlength}{292.19776999bp}%
    \ifx\svgscale\undefined%
      \relax%
    \else%
      \setlength{\unitlength}{\unitlength * \real{\svgscale}}%
    \fi%
  \else%
    \setlength{\unitlength}{\svgwidth}%
  \fi%
  \global\let\svgwidth\undefined%
  \global\let\svgscale\undefined%
  \makeatother%
  \begin{picture}(1,0.47554848)%
    \lineheight{1}%
    \setlength\tabcolsep{0pt}%
    \put(0,0){\includegraphics[width=\unitlength,page=1]{Schemaconnexite.pdf}}%
    \put(0.2907293,0.21313233){\color[rgb]{0,0.78431373,0}\makebox(0,0)[t]{\lineheight{1.25}\smash{\begin{tabular}[t]{c}$A$\end{tabular}}}}%
    \put(0.285385,0.39199353){\color[rgb]{0,0.78431373,0}\makebox(0,0)[t]{\lineheight{1.25}\smash{\begin{tabular}[t]{c}$A^{\prime}$\end{tabular}}}}%
    \put(0.53629103,0.39504426){\color[rgb]{0,0,0}\makebox(0,0)[t]{\lineheight{1.25}\smash{\begin{tabular}[t]{c}$K=(A^{\prime}\sqcup B^{\prime})^{c}$\end{tabular}}}}%
    \put(0.72707762,0.18746476){\color[rgb]{0.78431373,0,0}\makebox(0,0)[t]{\lineheight{1.25}\smash{\begin{tabular}[t]{c}$B$\end{tabular}}}}%
    \put(0.73456706,0.41766106){\color[rgb]{0.78431373,0,0}\makebox(0,0)[t]{\lineheight{1.25}\smash{\begin{tabular}[t]{c}$B^{\prime}$\end{tabular}}}}%
    \put(0,0){\includegraphics[width=\unitlength,page=2]{Schemaconnexite.pdf}}%
    \put(0.50355913,0.23798823){\makebox(0,0)[t]{\lineheight{1.25}\smash{\begin{tabular}[t]{c}$\alpha$\end{tabular}}}}%
    \put(0.4445238,0.14301832){\makebox(0,0)[t]{\lineheight{1.25}\smash{\begin{tabular}[t]{c}$\alpha/3$\end{tabular}}}}%
    \put(0.60622931,0.32782466){\makebox(0,0)[t]{\lineheight{1.25}\smash{\begin{tabular}[t]{c}$\alpha/3$\end{tabular}}}}%
    \put(0,0){\includegraphics[width=\unitlength,page=3]{Schemaconnexite.pdf}}%
    \put(0.79873591,0.28932333){\color[rgb]{0.78431373,0,0}\makebox(0,0)[t]{\lineheight{1.25}\smash{\begin{tabular}[t]{c}$\nu$\end{tabular}}}}%
    \put(0.13672401,0.22596611){\color[rgb]{0,0.78431373,0}\makebox(0,0)[t]{\lineheight{1.25}\smash{\begin{tabular}[t]{c}$\mu$\end{tabular}}}}%
  \end{picture}%
\endgroup%

%% file: Construction2Schema1.pdf_tex
\begingroup%
  \makeatletter%
  \providecommand\color[2][]{%
    \errmessage{(Inkscape) Color is used for the text in Inkscape, but the package 'color.sty' is not loaded}%
    \renewcommand\color[2][]{}%
  }%
  \providecommand\transparent[1]{%
    \errmessage{(Inkscape) Transparency is used (non-zero) for the text in Inkscape, but the package 'transparent.sty' is not loaded}%
    \renewcommand\transparent[1]{}%
  }%
  \providecommand\rotatebox[2]{#2}%
  \newcommand*\fsize{\dimexpr\f@size pt\relax}%
  \newcommand*\lineheight[1]{\fontsize{\fsize}{#1\fsize}\selectfont}%
  \ifx\svgwidth\undefined%
    \setlength{\unitlength}{653.95240123bp}%
    \ifx\svgscale\undefined%
      \relax%
    \else%
      \setlength{\unitlength}{\unitlength * \real{\svgscale}}%
    \fi%
  \else%
    \setlength{\unitlength}{\svgwidth}%
  \fi%
  \global\let\svgwidth\undefined%
  \global\let\svgscale\undefined%
  \makeatother%
  \begin{picture}(1,0.62016488)%
    \lineheight{1}%
    \setlength\tabcolsep{0pt}%
    \put(0,0){\includegraphics[width=\unitlength,page=1]{Construction2Schema1.pdf}}%
    \put(0.4992409,0.51129198){\makebox(0,0)[t]{\lineheight{1.25}\smash{\begin{tabular}[t]{c}Display\\zone\end{tabular}}}}%
    \put(0.76867271,0.50966416){\makebox(0,0)[t]{\lineheight{1.25}\smash{\begin{tabular}[t]{c}Computation\\zone\end{tabular}}}}%
    \put(0.53219197,0.36543488){\color[rgb]{1,0,0}\makebox(0,0)[t]{\lineheight{1.25}\smash{\begin{tabular}[t]{c}$t$\end{tabular}}}}%
    \put(0.77241247,0.3687509){\color[rgb]{1,0,0}\makebox(0,0)[t]{\lineheight{1.25}\smash{\begin{tabular}[t]{c}$\log (t)$\end{tabular}}}}%
    \put(0,0){\includegraphics[width=\unitlength,page=2]{Construction2Schema1.pdf}}%
    \put(0.27732861,0.3687509){\color[rgb]{1,0,0}\makebox(0,0)[t]{\lineheight{1.25}\smash{\begin{tabular}[t]{c}$\log (t)$\end{tabular}}}}%
    \put(0.20638818,0.51129198){\makebox(0,0)[t]{\lineheight{1.25}\smash{\begin{tabular}[t]{c}Format\\zone\end{tabular}}}}%
  \end{picture}%
\endgroup%

%% file: Construction2Schema2.pdf_tex
\begingroup%
  \makeatletter%
  \providecommand\color[2][]{%
    \errmessage{(Inkscape) Color is used for the text in Inkscape, but the package 'color.sty' is not loaded}%
    \renewcommand\color[2][]{}%
  }%
  \providecommand\transparent[1]{%
    \errmessage{(Inkscape) Transparency is used (non-zero) for the text in Inkscape, but the package 'transparent.sty' is not loaded}%
    \renewcommand\transparent[1]{}%
  }%
  \providecommand\rotatebox[2]{#2}%
  \newcommand*\fsize{\dimexpr\f@size pt\relax}%
  \newcommand*\lineheight[1]{\fontsize{\fsize}{#1\fsize}\selectfont}%
  \ifx\svgwidth\undefined%
    \setlength{\unitlength}{370.37870632bp}%
    \ifx\svgscale\undefined%
      \relax%
    \else%
      \setlength{\unitlength}{\unitlength * \real{\svgscale}}%
    \fi%
  \else%
    \setlength{\unitlength}{\svgwidth}%
  \fi%
  \global\let\svgwidth\undefined%
  \global\let\svgscale\undefined%
  \makeatother%
  \begin{picture}(1,1.09427709)%
    \lineheight{1}%
    \setlength\tabcolsep{0pt}%
    \put(0,0){\includegraphics[width=\unitlength,page=1]{Construction2Schema2.pdf}}%
    \put(0.84792148,0.60011628){\color[rgb]{1,0,0}\makebox(0,0)[t]{\lineheight{1.25}\smash{\begin{tabular}[t]{c}$T_n$\end{tabular}}}}%
    \put(0,0){\includegraphics[width=\unitlength,page=2]{Construction2Schema2.pdf}}%
    \put(0.84792148,0.20727486){\color[rgb]{0,0,0}\makebox(0,0)[t]{\lineheight{1.25}\smash{\begin{tabular}[t]{c}$T_{n-1}$\end{tabular}}}}%
    \put(0,0){\includegraphics[width=\unitlength,page=3]{Construction2Schema2.pdf}}%
  \end{picture}%
\endgroup%